\documentclass{article}

\usepackage{arxiv}

\usepackage[hyphens]{url}
\usepackage{subfiles}
\usepackage{lipsum}
\usepackage{graphicx}
\usepackage{ytableau}
\usepackage{todonotes}

\usepackage[utf8]{inputenc}
\usepackage[T1]{fontenc}

\usepackage{comment}
\usepackage{nicefrac}   
\usepackage{amsmath}
\usepackage{bbm}
\usepackage{bm}

\usepackage{amsthm}
\usepackage{newtxtext,newtxmath}
\usepackage{datetime}

\newtheorem{theorem}{Theorem}[section]
\newtheorem{corollary}{Corollary}[section]
\newtheorem{lemma}{Lemma}[section]
\newtheorem{proposition}{Proposition}[section]

\newtheorem{definition}{Definition}[section]

\newcommand{\ps}[1]

\usepackage[
    backend=biber,natbib,
    style=numeric
]{biblatex}
\author{{Yuhui Jin} \\
	Department of Mathematics\\
	California Institute of Technology\\
	Pasadena, CA 91106 \\
	\texttt{yjjin@caltech.edu} \\
	}

\date{}

\renewcommand{\headeright}{}

\begin{document}
\title{On the Hecke Module of $\text{GL}_n(k[[z]])\backslash \text{GL}_n(k((z)))/\text{GL}_n(k((z^2)))$}

\maketitle 
\author{Yuhui Jin}

\begin{abstract}

Every double coset in $\text{GL}_m(k[[z]])\backslash \text{GL}_m(k((z)))/\text{GL}_m(k((z^2)))$ 
is uniquely represented by a block diagonal matrix with diagonal blocks in $\{1,z,
\begin{pmatrix}
1& z\\
0 &z^i \\

\end{pmatrix} (i>1)\}$ 
if $char(k) \neq 2$ and $k$ is a finite field. 
These cosets form a (spherical) Hecke module $\mathcal{H}(G,H,K)$ over the (spherical) Hecke algebra $\mathcal{H}(G,K)$ of double cosets in $K\backslash G/H$, where $K=\text{GL}_m(k[[z]])$ and $H=\text{GL}_m(k((z^2)))$ and $G=\text{GL}_m(k((z)))$. Similarly to Hall polynomial $h_{\lambda,\nu}^{\mu}$ from the Hecke algebra $\mathcal{H}(G,K)$, coefficients  $h_{\lambda,\nu}^{\mu}$ arise from the Hecke module. We will provide a closed formula for $h_{\lambda,\nu}^\mu$, under some restrictions over ${\lambda,\nu,\mu}$. 
\end{abstract}

\tableofcontents

\section{Introduction}
\subsection{Symmetric Elliptic Difference Equation}
In \cite{rains2013generalized, rains2019birational,rains2016noncommutative}, Rains studies a general elliptic difference equation of the form $v(q+z)=A(z)v(z)$ where $q$ is a point of an elliptic curve and $A(z)$ is a matrix of elliptic functions with $\det(A(z))$ not identically 0. Such equations arise, for instance, in studying elliptic analogues of ordinary and $q$-hypergeometric special functions. At the elliptic level, these functions tend to satisfy an additional symmetry of the form $v(z)=v(-z)$, which leads the equation itself to satisfy an additional symmetry, namely that $A(-q-z)=A(z)^{-1}$. To understand symmetries of such equations, it is important to understand how they behave {\em locally}, i.e., over the ring of formal power series. At a typical point $z_0$, the equation is regular if and only if $A$ and $A^{-1}$ are both holomorphic at $z_0$, since then the spaces of solutions near $z_0$ and near $q+z_0$ are naturally isomorphic. However, for symmetric equations, we care about the space of {\em symmetric} solutions, which means that the condition is more subtle. (For instance, the equation $v(q+z)=-v(z)$ is symmetric, but any holomorphic solution satisfying $v(-z)=v(z)$ will vanish at points with $2z=-q$.) We thus see that to understand local singularities in complete generality, we need to better understand matrices $A$ satisfying this symmetry. Two such matrices give the same local behavior at $z_0$ (satisfying $2z_0=-q$) if and only if they are related by an invertibly holomorphic change of basis at $z_0$, and thus we need to understand such matrices up to (twisted) conjugation by invertibly holomorphic matrices.

Instead of the difference equation, it is natural to classify the matrices $A$. 
We can rephrase the problem in an abstractly geometric way. Given an elliptic curve $C_\alpha$ over algebraically closed field $k$ and a translation $\tau_q: C_{\alpha}\rightarrow C_{\alpha}$ and $\eta:z\to -q-z$, 
the problem of classifying $A$ becomes classifying matrices $\text{GL}_n(k(C_{\alpha}))$ such that 
$\eta^*A=A^{-1}$. By Hilbert's theorem 90, we have the following proposition.
\begin{proposition}

\cite[Proposition~2.1]{rains2013generalized}
Let $L/K$ be a quadratic field extension, and let $A\in \text{GL}_n(L)$ be a matrix such that $\bar{A}=A^{-1}$, where $\bar{.}$ is the conjugation of $L$ over $k$. Then there exists a matrix $B\in GL_n(L)$ such that $A=\bar{B} B^{-1}$ and $B$ is unique up to right multiplication by $\text{GL}_n(K)$.
\end{proposition}

Thus, there exist a canonical factorization $A=\eta^*B^{-t} B^{t}$ where $B$ is an injective morphism $B: \pi_{\eta}^* V\rightarrow \mathcal{O}_{C_\alpha}^n$, with $V$ a rank n vector bundle over $\mathbb{P}^1$, the projective space. Take any point in $p\in C_\alpha$, $B$ is a matrix over the local ring. 
In \cite[Chapter~8]{rains2013generalized} , Rains first mentioned the following without proof: if $p$ is fixed by $\eta $, the invariants $\lambda(B;p)$ is determined by the equivalence relation of left multiplication by invertible matrices over the local ring and right multiplication by symmetric matrices over the local field. That is, $B$ is a direct sum of matrices $1,z$, and $\begin{pmatrix}
1& z\\
0 &z^i \\

\end{pmatrix} (i>1)$. In Chapter 2, we will give a proof in a different way from how Rains observed it.

In \cite[Chapter~13.2]{rains2016noncommutative}, Rains constructs the solution sheaf with respect to a vector bundle over $C_\alpha$, a system of local condiitions, and $q$-connections over $A$ for an elliptic difference equation and uses the dominant coweight to study the multiplication of $A(q^{m-1}z)A(q^{m-2}z)...A(z)$, in other word, the interactions of poles and zeros for shift matrices $A$. In the symmetric case, we will study the interactions of $A$ and $B$ in terms of poles and zeros.

\subsection{Hecke Algebra and Hall Polynomials}
Let $k$ be a finite field of $q$ elements with characteristic $char(k) \ne 2$. Let  $k[[z]]$ be the ring of formal power series. Let $k((z))$ be the the field of fractions of the ring $k[[z]]$  of formal power series. Let $G=\text{GL}_n(k((z)))$ and let $K=\text{GL}_n(k[[z]])$ be the maximal compact subgroup of $G$. Then $G/K$ is the affine Grassmaniann.
\begin{theorem}
\cite[Chapter~9]{bump} 
\label{cartan}
Every double coset in $K\backslash G/K$ has a unique representation of the form diag $(z^{\lambda_1}$
$,z^{\lambda_2}$
$,...,z^{\lambda_n})$, where $\lambda_1\geq \lambda_2\geq ...\geq \lambda_n.$ We define the partition $(\lambda_1,\lambda_2,...\lambda_n)$ to be the dominant coweight of all elements in the corresponding double coset and denote it by $\rho$.

\end{theorem}

This theorem is known as the \emph{Cartan decomposition}. 
\begin{definition}
For $g\in G$, define gcd(g) to be the fractional ideal of $k[[z]]$ generated by the entries in $g$. Let $\wedge^a$ be the a-th exterior power representation; the matrix entries of $\wedge^a(g)$ are the a by a minors of $g$. Then $gcd(\wedge^a g )$ is the fractional ideal generated by these minors.
\end{definition}
The fractional ideal $gcd(\wedge^ag)$ is invariant under left and right multiplication by $K$.
In the proof of Theorem~\ref{cartan}, we have the following statement.

For any $g\in G$ with $\rho(g)=(\lambda_1,\lambda_2,...\lambda_n)$, we have $gcd(g)=z^\lambda_n\text{ and } gcd(\wedge^i g)=z^{\sum_{j=1}^{i} \lambda_{n-j+1}}$ for all i.

The (spherical) Hecke algebra $\mathcal{H}(G,K)$ is the convolution algebra of all complex-valued continuous compactly supported K-bi-invariant functions on G. Let $c_\lambda$ be the characteristic function of $Kz^\lambda K$. The Hecke algebra $\mathcal{H}(G,K)$ has a basis given by $\{c_\lambda\}_{l(\lambda)= n}$. In fact $c_{\mu}*c_{\nu}=\sum_{\lambda} g_{\mu\nu}^{\lambda}(q)c_\lambda$, where $ g_{\mu\nu}^{\lambda}(q)$ is the Hall polynomial. 
\begin{proposition}
\label{satake}
 \cite[Proposition~37]{bump} (special case of Satake Isomorphism)

Let $\theta_r=q^{-r(n-r)/2}c_{1^r}$ for all r. The ring structure of $\mathcal{H}(G,K)$ is a polynomial ring over $\theta_1,\theta_2,...,\theta_n,\theta_{n}^{-1}.$
\[\mathcal{H}(G,K)\cong \mathbbm{C}[\theta_1,\theta_2,...,\theta_n,\theta_{n}^{-1}]\]

\end{proposition}

\begin{proposition}
\label{HL}
\cite[Chapter~5, Theorem~2.7]{macdonald1998symmetric}

Let $G^+=G\cap M_n(k[[z]]).$ Then, naturally we have $\mathcal{H}(G,K)=\mathcal{H}(G^+,K)[c^{-1}_{1^n}].$
There is a $\mathbbm{C}$-algebra isomorphism  from $\mathcal{H}(G^+,K)$ to  the ring of symmetric polynomials in n variables with coefficients in $\mathbbm{C}[q^{-1}]$, i.e., $f\colon c_{\lambda}\to q^{-n(\lambda)}P_\lambda(x_1,...x_n;q^{-1}).$ Here $P_\lambda$ is the Hall-Littlewood polynomial.
\end{proposition}

In 2007, Rains and Vazirani \cite{rains2007vanishing} developed affine Hecke algebra techniques to prove results in terms of vanishing integrals of Macdonald and Koornwinder polynomials.
However, at q = 0 (the Hall–Littlewood level), these approaches do not work, although one can obtain the results by
taking the appropriate limit. In \cite{venkateswaran2012vanishing}, Venkateswaran developed a p-adic representation
theory approach dealing with this special case.

The Hall polynomial was historically studied by Hall, which is an important building block in the theory of special functions.  MacDonald \cite[Chapter~2]{macdonald1998symmetric} provides an extensive introduction of the Hall polynomial. Since the actual explicit formula is lengthy and involves LR-sequence of type $(\mu',\nu';\lambda')$, we will only mention some results that are used in this thesis:
\begin{itemize}
\item
$g_{(1^r)\nu}^{\lambda}(q)=0$ unless $\lambda-\nu$ is a vertical stripe of length r. (Pieri rule)
\item
$g_{(r)\nu}^{\lambda}(q)=0$ unless $\lambda-\nu$ is a horizontal stripe of length r. (Dual Pieri rule)

\end{itemize}

The product $P_{\nu}P_{\mu}$ of Hall-Littlewood polynomials is a linear combination of the $P_{\lambda}$'s for all $\lambda$'s with $|\lambda|=|\nu|+|\mu|: P_{\nu}(x;t)P_{\mu}(x;t)=\sum_{\lambda}g_{\mu\nu}^{\lambda}(t)P_{\lambda}(x;t)$.

\subsection{Outline of the Paper}
In the second chapter, we give a proof of the decomposition of $K\backslash G/H$, where $G =\text{GL}_n(k((z))), H =\text{GL}_n(k((z^2))), \text{ and } K = \text{GL}_n(k[[z]]),$ and the Pieri rule as well as the dual Pieri rule of the interaction between dominant coweight and symmetric coweight. In the third chapter, a Hecke module will be defined and we will compute the structural constant $h_{-1^r \nu}^\lambda$.  In the fourth chapter we will compute the structural constant $h_{-r\nu}^{\lambda}$.

\subsection{Future Directions}
Several questions can be raised as future directions. In terms of  double coset structure $K\backslash G/H$, if we replace the maximal compact subgroup K with an Iwahori subgroup, i.e. a conjugate of the inverse map of the subgroup of all upper triangular matrices in $\text{GL}_n(k)$, are there any good coset representatives and a nice Hecke module that could be possibly extended to other types?  In Chapter 3, we provide a formula for $h_{-1^r \nu}^\lambda$. Thus, a natural question to ask is there a pure combinatorial approach to compute such $h_{-1^r \nu}^\lambda$, since in this part, the building block of $h_{-1^r a^n}^\lambda$ is obtained from induction and the simplicity of the formula itself shines a light on a possible purely combinatorial proof. Furthermore, what is the complete formula for $h_{\mu,\nu}^\lambda$ and, if possible, is there any construction we can make to give a one-to-one correspondence from the Hecke module structure to the ring of polynomials (similarly as the correspondence from Hall-Littlewood polynomials to the Hecke algebra)? 
\subsection{Notation}
Recall the $q$-integer $[n]=\frac{q^n-1}{q-1}$ and $q$-factorial $[n]!=[n][n-1]....[1].$ The $q$-binomial is defined as
$ \binom{n}{k}=\frac{[n]!}{[n-k]![k]!}$ and  
the $q$-multinomial is defined as \[\binom{n}{ a_1,a_2,...a_i}=\frac{[n]!}{[a_1]![a_2]!...[a_i]![n-a_1-a_2-...-a_i]!}.\] 
For the simplicity of notation in Chapters 3 and 4, we denote $\binom{n}{ a_1,a_2,...a_i}=0$  if any $a_j$ is smaller than 0 and same for the $q$-binomial. In this thesis, ALL terms in the form of $\binom{n}{k},\binom{n}{ a_1,a_2,...a_i}$ are $q$-binomial or $q$-multinomial. 

In this thesis, we will use English notation for Young Diagram.

Notation on partitions: A \emph{partition} is a sequence $\lambda=(\lambda_1,\lambda_2,...)$ of integers in decreasing order (we allow negative integers) and containing finitely many nonzero terms. For the sake of simplicity, we may neglect the string of zeros when writing explicitly the sequence. The definition of length for $\lambda$ is also different from the tradition: Here $l(\lambda)$ is the number of all terms in $\lambda$. Furthermore, we define $m_i(\lambda)=|\{j:\lambda_j=i\}|$, and $n(\lambda)=\sum_{i>0} (i-1) \lambda_i$. The partition $\lambda'$ is the conjugate of $\lambda$. 

Notation for matrices: Let $\mathbbm{1}_{(i,j)}(f)$ with $i\ne j$ be a transvectional matrix: All entries in diagonal are $ 1$ and the only nonzero off diagonal term are at row $i$ and column $j$ with entry $f$. Let $\mathbbm{1}_{(i,i)}(f)$ be the diagonal matrix: All entries in diagonal except for row $i$ is $1$ and the $(i,i)$-entry is $f$.

Notation for Laurent series: Since we require the base field to be finite, $k((z))$ is a local field. Given $f=\sum_{i\geq r} a_i z^i$ with $a_r\neq 0$, the valuation of $f$, $v(f)$, is $r$. Let $f_{odd}$ be the odd part of $f $ and $f_{even}$  be the even part of $f$, that is, $f_{even}+f_{odd}=f$ and $f_{even},f_{odd}/z\in k((z^2)).$

\section{Symmetric Coweight}
\subsection{Decomposition}

\begin{theorem}
\label{start_decomp}
Every double coset in $\text{GL}_m(k[[z]])\backslash \text{GL}_m(k((z)))/\text{GL}_m(k((z^2)))$ is uniquely represented by block diagonal matrix whose diagonal blocks are  $1$, $z$, or $
\begin{pmatrix}
1& z\\
0 &z^i \\

\end{pmatrix}$, where  $i>1$ and the base field k is finite and has $char(k)\ne 2$ .

\end{theorem}

\begin{proof}

By the Iwasawa decomposition, for any $g\in \text{GL}_m(k((z)))$, there exists $h\in \text{GL}_m(k[[z]]))$ such that $gh$ is upper triangular.

Before we start the actual proof of the theorem, we introduce 4 tricks that are used in the actual proof. We call left multiplication by matrices in $\text{GL}_m(k[[z]])$ and right multiplication by matrices in $\text{GL}_m(k((z^2)))$ allowed operations .

Trick 1: A submatrix of the form $ \begin{pmatrix}
z^b &f_e z^a\\
0  &f_1 \\

\end{pmatrix}$ (the column with element $z^b$ only has one nonzero entry and $f_e\in k((z^2))$ , $v(f_e)=0$) can be reduced to $\begin{pmatrix}
z^b &z^a\\
0  &f_1 \\

\end{pmatrix}$  with allowed operations. 
We left multiply by $\mathbbm{1}_{(1,1)}(1/f_e)$,  right multiply by $\mathbbm{1}_{(1,1)}(f_e)$.
Note that this may change the other entries in the same row as the top row of the submatrix.

Trick 2: If $d\leq a\text{ and for the two pairs } \{a,d\} \text{ and }\{b,c\} $, elements in each pair are of different parities, a submatrix of the form $\begin{pmatrix}
z^b & 0& z^a\\
0  & z^c&z^d \\

\end{pmatrix}$  (the column with  $z^b$ only has one nonzero entry) can be reduced to $\begin{pmatrix}
z^b & 0& 0\\
0  & z^c&z^d \\

\end{pmatrix}$   with allowed operations. 
We left multiply by $\mathbbm{1}_{(1,2)}(-z^{a-d})$, right multiply by $\mathbbm{1}_{(1,2)}(z^{a-d+c-b})$.
Note that this may change the other entries in the same row as the top row of the submatrix.

Trick 3:
A submatrix of the form $\begin{pmatrix}
1 &f\\
0  &h \\
\end{pmatrix}$ (the column with the 1 only has one nonzero entry) can be reduced to  $\begin{pmatrix}
1 & z^{v(f_{\text{odd}})}\\
0  &h \\
\end{pmatrix}$ or $\begin{pmatrix}
1 &0\\
0  &h \\
\end{pmatrix}$ where ${0,v(f_{\text{odd}})}$ are of different parity  with allowed operations. Denote $f=f_{\text{odd}}+f_{\text{even}}$.
We first right multiply by  $\mathbbm{1}_{(1,2)}(-f_{\text{even}})$. If $f_{\text{odd}}=0$, then the second submatrix is achieved. Otherwise, by Trick 1, $f_{\text{odd}}$ reduces to $z^{v(f_{\text{odd}})}$. Note that this may change the other entries in the same row as the top row of the submatrix. If the (1,1) term is $z$, the trick is similar.

Trick 4: 
A submatrix of the form $\begin{pmatrix}
1 &  z& f\\
0  & z^a&g \\

\end{pmatrix}$  (two columns do not have other nonzero entries in the matrix) can be reduced to $\begin{pmatrix}
1 &  z& 0\\
0  & z^a&g-f_{\text{odd}}z^{a-1} \\

\end{pmatrix}$ with allowed operations.
We right multiply by $\mathbbm{1}_{(1,3)}(-f_{\text{even}})$$\mathbbm{1}_{(2,3)}(-f_{\text{odd}}/z)$.

Given the above reductions, the proof proceeds by induction on $m$.  Suppose that the claim holds for matrices of smaller dimension. In particular, the claim applies to the upper left $m-1\times m-1$ submatrix of $X$, which allows us to assume without loss of generality that submatrix is block diagonal with blocks of the desired form.

We use trick 3,4 to modify entries in column $m$ of $X$: $X_{i,m}\text { is in the form of }z^a$  if $X_{i,i}$ corresponds to one-dimensional blocks (trick 3). $X_{i,m}=0$ if $X_{i,i}$ corresponds to the first row of a two dimensional part (trick 4). If there exists $ p $ such that $v(X_{p,m}) \geq v(X_{m,m})$  and $X_{p,m}$ is nonzero, $X_{p,m}$ can be reduced to 0 by row operations on row $m$ and $p$. 
Notice that to prove the inductive statement, it suffices to reduce one nonzero entry $X_{i,m}$ to be 0 or completely separate one block from the matrix. (Then by induction, matrices of smaller dimension are block diagonalizable.) Therefore, one more assumption on valuations of column m entries can be added: $X_{m,m}$ has the largest valuation among all nonzero elements in column $m$.  We permute rows and columns in the first $m-1$ by $m$ submatrix such that 
for all nonzero $X_{i,m}$,  $v(X_{i,m})$ is in increasing order and the first $m-1$ by $m-1$ submatrix is still block diagonal. There will be two cases according to the permutation: Case 1 is the first block in a two-dimensional block (i.e., an entry in column m with the least valuation corresponds to a 2-dimensional block); Case 2 is the first block in a one-dimensional block (i.e., an entry in column $m$ with the least valuation corresponds to a one-dimensional block).

Case 1: A nonzero column $m$ entry with the least valuation corresponds to a two-dimensional block. We will show that  allowed operations can make a one-dimensional block completely decomposed from the matrix.

The following is a submatrix of $\{1,2,m\}$ rows and columns: 
$\begin{pmatrix}
1 & z & 0  \\
0 & z^a &   z^{i_1}f_{1}\\
0 & 0   &  z^b  \\
\end{pmatrix}.$

Step 1: Make the second entry of column $m$ to be the only nonzero entry by row operations:
$
\begin{pmatrix}
1 & z  & 0  \\
0 & z^a & z^{i_1}f_{1}\\
0 & -z^{b+a-i_1}f^{-1}_{1}    & 0  \\
\end{pmatrix}.$

Step 2: Multiply the second row by $f_{1}^{-1}$ and make the second entry of column two in the form $z^xf_e, \text{ where }f_e\in k((z^2))$  by column operations (column m only has one nonzero element).

$\begin{pmatrix}
1 & z  & 0  \\
0 & z^af^{-1}_{1}   & z^{i_1}\\
0 & -z^{a+b-i_1}f^{-1}_{1}  & 0  \\
\end{pmatrix}
\rightarrow
\begin{pmatrix}
1 & z   & 0  \\
0 & z^{a'}f_{e3}   & z^{i_1}\\
0 & -z^{a+b-i_1}f^{-1}_{1}  & 0  \\
\end{pmatrix},$

$\text{where }  z^af^{-1}_{1}= z^{i_1}f_{e2} + z^{a'} f_{e3} \text{ and }\{a',i_1\}\text{ have different parity},\text{ and } f_{e3},f_{e2}\in k((z^2))$, $v(f_{e3})=0.$ If $f_{e3}=0$, $z^{i_1}$ is a one-dimensional  block in the submatrix.

Step 3: Turn $X_{2,2}$ to the form $z^x$. (We multiply the second column by $f^{-1}_{e3}$, first row by $f_{e3}$, first column by $f^{-1}_{e3}$). 

$
\begin{pmatrix}
1 & z  & 0  \\
0 & z^{a'}  & z^{i_1}\\
0 & -z^{a+b-i_1}f^{-1}_{1}f^{-1}_{e3}   & 0  \\

\end{pmatrix}$

Step 4: Make the only nonzero entry in the second row be $X_{2,m}$.

$ 
\begin{pmatrix}
1 & z  & 0  \\
z^{a'-1}  &0  & z^{i_1}\\
0 & -z^{a+b-i_1}f^{-1}_{1}f^{-1}_{e3}    & 0  \\

\end{pmatrix} \rightarrow 
\begin{pmatrix}
1 & z   & 0  \\
0 &0 & z^{i_1}\\
0 & -z^{a+b-i_1}f^{-1}_{1}f^{-1}_{e3}  & 0  \\
\end{pmatrix}$ 

$\text{From the inequality of } a'\geq a\geq 2 \text{ and the fact that } \{a'-1, i_1\}\text{ are of same parity}$, trick 2 removes element $z^{a'-1}$.
Thus, $z^{i_1}$ is a one-dimensional block. We can reduce it to either 1 or $z$.

Case 2: A nonzero column $m$ entry with the least valuation corresponds to a one-dimensional block. 
If there exists at least 2 one-dimensional blocks in the $m-1$ by $m-1$ submatrix, assume that one of them is in the $\ell$-th row. Then $X_{\ell,m}$ reduces to 0 by trick 2. Therefore, $X_{\ell,\ell}$ is the only nonzero entry in column or row $\ell$ in the matrix. Thus, for this case, it remains to show that if all blocks from row 2 to $m-1$ are 2-dimensional, we can still split one block out of the matrix. (If there is no two-dimensional block, which is equivalent to m=2, the 2 by 2 matrix is a block.)

We will have two major steps for this part:

1. We simplify our matrix into the following form: 
on column $m$ of $X$, $X_{2i,m}=0$ ($X_{2i,2i}$ corresponds to the first row of a two-dimensional block), $X_{2i+1,m}$ is in the form of $z^a$ ( $X_{2i+1,2i+1}$ corresponds to the second row of a two-dimensional block), and $X_{1,m}$ is still in the form of $z^a$.

$\begin{pmatrix}
z^{a_{1,1}} &  0 &0 &0&\cdots &z^{a_{1,2}}  \\
0 &z^{a_{2,1}}&z^{a_{2,1}} &0 & \cdots &0\\
0 &0&z^{a_{3,1}}&0 & \cdots &z^{a_{3,2}}\\
\cdots\\
0 &0&0&0 & \cdots &z^{a_{n,1}}\\
\end{pmatrix}$

2. We completely separate one block from the matrix.

Part 1: Step 1:
Write $X_{3,m}=z^af_1$. If $f_1\in k[[z^2]]$, then, by trick $1$, we have  $X_{3,m}\text{ go to }z^a$. If $f_1\notin k[[z^2]]$, multiplying $f_1^{-1}$ to the third row will take $X_{3,m}$ to $z^a$ and $X_{3,3}\text{ to }z^kf_1^{-1}$. (The following matrix illustration is the submatrix of $\{1,2,3,m\}$ rows and columns.)

$
\begin{pmatrix}
z^w & 0 &0& z^{i}  \\
0  & 1&z   &0\\
0 &0 & z^k  & z^a f_{1}  \\
0& 0&0 & z^l\\
\end{pmatrix}
\rightarrow
\begin{pmatrix}
z^w & 0 &0& z^{i}  \\
0  & 1&z   &0\\
0 &0 & z^kf_{1}^{-1}  & z^a  \\
0& 0&0 & z^l\\
\end{pmatrix}$

Step 2: Row operations on row 2,3 can split $X_{3,3}$ into even and odd parts to $X_{3,2}, X_{3,3}$ (we have $z^kf^{-1}_{1}=z^{k'}f_{e1}-z^{a'}f_{e2}$, with $v(f_{e1})=v(f_{e2})=0$ and $\{a',a\}$ are of different parity), i.e., $X_{3,3}\to -z^{a'}f_{e2}, X_{3,2} \to -z^{k'}f_{e1} $. Then, we multiply column 3 by $f_{e2}^{-1}$, multiply row 2 by $f_{e2}$, multiply column 2 by $f_{e2}^{-1}$. That is, $f_{e2}^{-1}$ is in $X_{2,2}$(the matrix on the left).
Notice $2|k'-1+a'$. Thus, by column operation on column 2,3 we eliminate $X_{3,2}$ (the matrix on the right).

$
\begin{pmatrix}
z^w & 0 &0   & z^{i}  \\
0  & 1&z  &0\\
0&-z^{k'-1}f_{e1}f_{e2}^{-1} &-z^{a'}   & z^a \\

0& 0  &0  &  z^{l} \\
\end{pmatrix}
\rightarrow
\begin{pmatrix}
z^w & 0 &0   & z^{i}  \\
0  & 1-z^{k'-a'}f_{e1}f^{-1}_{e2}&z  &0\\
0&0&-z^{a'}   & z^a \\

0& 0  &0  &  z^{l} \\
\end{pmatrix}
$

Step 3: We write $z^uf_{2}=1-z^{k'-a'}f_{e1}f^{-1}_{e2} (v(f_{2})=0)$. Multiplying row 2 by $f_{o2}^{-1}$, column operations on column 2,3 will make $X_{2,3}/z^{v(X_{2,3})}\in k((z^2)) $. We write  $zf^{-1}_{2}=z^{u'}f_{eu1}+z^{u''}f_{eu2}$ where $2|u-u' \text{ and } u',u''$are of different parity and $v(f_{eu2})=0$. By column operations on columns 2,3, $X_{2,3}$ would be reduced to $z^{u''}f_{eu2}$. Next by trick 1, the term $f_{eu2}$ and minus sign are eliminated.

$
\begin{pmatrix}
z^w & 0 &0  & z^{i}  \\
0  &z^u& zf_{e2}^{-1}  &0\\
0 & 0&-z^{a'}   & z^a \\
0& 0  & 0 & z^l\\
\end{pmatrix}\rightarrow
\begin{pmatrix}
z^w & 0 &0  & z^{i}  \\
0  &z^u& z^{u''}f_{eu2}  &0\\
0 & 0&-z^{a'}   & z^a \\
0& 0  & 0 & z^l\\
\end{pmatrix}
\rightarrow
\begin{pmatrix}
z^w & 0 &0  & z^{i}  \\
0  &z^u& z^{u''} &0\\
0 & 0&z^{a'}   & z^a \\
0& 0  & 0 & z^l\\
\end{pmatrix}$

Now, redoing steps 1-4 for all two-dimensional blocks such that all nonzero entries in column $m$ are in the form of $z^a$ would finish part 1.

Part 2. Taking any nonzero elements to be 0 will break the matrix into smaller blocks. Recall that there are two semi-block forms from step 1 above:
\[
\begin{pmatrix}
z^{u} & z^{u''} &\cdots& 0\\
0 & z^{a'} &\cdots& z^a\\
\end{pmatrix}
\text{ and }
\begin{pmatrix}
1 & z &\cdots& 0\\
0 & z^{k} &\cdots& z^a\\
\end{pmatrix}.
\]
The former is the semi-block obtained from part 1 and the latter is the case where $f_1 \in k[[z^2]]$. The former has a restriction: For the two pairs  $\{u,u''\}\text{ and }\{a,a'\}$, elements in each pair are of different parities. The latter has a restriction: $k>1.$ 

If $u''\geq a'$, trick 2 takes the former semi-block to $
\begin{pmatrix}
z^{u} &0 &\cdots& 0\\
0 & z^{a'} &\cdots& z^a\\
\end{pmatrix}$ .

For $\begin{pmatrix}
z^{u} & z^{u''} &\cdots& 0\\
0 & z^{a'} &\cdots& z^a\\
\end{pmatrix}$, we rewrite it as $\begin{pmatrix}
z^{u''} & z^{u} &\cdots& 0\\
0 & z^{a'+u-u''} &\cdots& z^a\\
\end{pmatrix}$. Similarly for $\begin{pmatrix}
1 & z &\cdots& 0\\
0 & z^{k} &\cdots& z^a\\
\end{pmatrix}$, if $k,a$ are of different parity, we rewrite it as $\begin{pmatrix}
z & 1 &\cdots& 0\\
0 & z^{k-1} &\cdots& z^a\\
\end{pmatrix}$.
Now, two semi-blocks are both of the form $\begin{pmatrix}
z^w & z^{w'} &\cdots& 0\\
0 & z^{b} &\cdots& z^{b'}\\
\end{pmatrix}$, where for the two pairs  $\{w,w'\}\text{ and }\{b,b'+1\}$, elements in each pair are of different parities and $w'<b$.

We claim that on the submatrix of $\{1,2,3,m\}$ rows and columns, by those allowed operations, we can extract one two-dimensional block from the matrix and obtain
$\begin{pmatrix}
z^{a_{1,1}} & 0 &0 & z^{a_{1,2}}  \\
0 &z^{a_{2,1}}&z^{a_{2,2}}  &0\\
0 &0&z^{a_{3,1}}&z^{a_{3,2}}\\
0 &0&0 &z^{a_{n,1}}\\
\end{pmatrix}.$
Let $\text{ condition } (\mathbf{A'})$  denote
$a_{3,1}-a_{2,2}\leq a_{3,2}-a_{1,2} $, and $\text{ condition } (\mathbf{B'})$ denote $a_{3,1}-a_{2,2}> a_{3,2}-a_{1,2}$.

$\mathbf{A'}$: 
$\begin{pmatrix}
z^{a_{1,1}} & 0 &0 & z^{a_{1,2}}  \\
0 &z^{a_{2,1}}&z^{a_{2,2}}  &0\\
0 &0&z^{a_{3,1}}&z^{a_{3,2}}\\
0 &0&0 &z^{a_{n,1}}\\
\end{pmatrix}
\rightarrow
\begin{pmatrix}
z^{a_{1,1}} & 0 &0  & z^{a_{1,2}}  \\
0 &z^{a_{2,1}}&z^{a_{2,2}}  & z^{a_{3,2}+a_{2,2}-a_{3,1}}\\
0 &0&z^{a_{3,1}}  &0\\

0 &0&0 &z^{a_{n,1}}\\
\end{pmatrix}
$

$\rightarrow
\begin{pmatrix}
z^{a_{1,1}} & 0 &0  & z^{a_{1,2}}  \\
z^{a_{3,2}+a_{2,2}-a_{3,1}+a_{1,1}-a_{1,2}} &z^{a_{2,1}}&z^{a_{2,2}}  & 0\\
0 &0&z^{a_{3,1}}  &0\\

0 &0&0 &z^{a_{n,1}}\\
\end{pmatrix}
\rightarrow
\begin{pmatrix}
z^{a_{1,1}} & 0 &0  & z^{a_{1,2}}  \\
0 &z^{a_{2,1}}&z^{a_{2,2}}  & 0\\
0 &0&z^{a_{3,1}}  &0\\

0 &0&0 &z^{a_{n,1}}\\
\end{pmatrix}
$

Condition $\mathbf{B'}$ is more subtle than $\mathbf{A'}$: First we do row operations from row 3 to $5,7,\ldots,$ and $m$ to remove all entries in $X_{2k+1,m}$ for $k>1$ and $X_{m,m}$. That is, terms on column $m$ are transferred to column 3. In the following submatrices  we rewrite $a_{n,1}+a_{3,1}-a_{3,2}\text{ to be } a'_{n,1}$.

$\begin{pmatrix}
z^{a_{1,1}} & 0 &0 & z^{a_{1,2}}  \\
0 &z^{a_{2,1}}&z^{a_{2,2}}  &0\\
0 &0&z^{a_{3,1}}&z^{a_{3,2}}\\
0 &0&z^{a'_{n,1}} &0\\
\end{pmatrix}
\rightarrow
\begin{pmatrix}
z^{a_{1,1}} & 0 &z^{a_{3,1}-a_{3,2}+a_{1,2}} & z^{a_{1,2}}  \\
0 &z^{a_{2,1}}&z^{a_{2,2}}  &0\\
0 &0&0&z^{a_{3,2}}\\

0 &0&z^{a'_{n,1}} &0\\
\end{pmatrix}$

$\rightarrow
\begin{pmatrix}
z^{a_{1,1}} & z^{a_{3,1}-a_{3,2}+a_{1,2}+a_{2,1}-a_{2,2}}  &0 & z^{a_{1,2}}  \\
0 &z^{a_{2,1}}&z^{a_{2,2}}  &0\\
0 &0&0&z^{a_{3,2}}\\

0 &0&z^{a'_{n,1}} &0\\
\end{pmatrix}
\rightarrow
\begin{pmatrix}
z^{a_{1,1}} & 0  &0 & z^{a_{1,2}}  \\
0 &z^{a_{2,1}}&z^{a_{2,2}}  &0\\
0 &0&0&z^{a_{3,2}}\\
0 &0&z^{a'_{n,1}} &0\\
\end{pmatrix}
$

Now  rows 1,3 and columns 1,4 form a two-dimensional block in the matrix (column 3 still holds all the original column $m$ information). 

The remaining of this proof is to show such blockwise decomposition is unique. 

Given $B(z)\in \text{GL}_n(k((z)))$, we write a decomposition of $B$ as $B(z)=g\Lambda h$, where $g\in \text{GL}_n(k[[z]])$. $h\in \text{GL}_n(k((z^2)))$. Denote $\tilde{B}(z)=B^{-1}(-z)\text{ and }\hat{B}=B\tilde{B}$.
For $C=\begin{pmatrix}
1&z\\
0&z^i\\
\end{pmatrix}$,
the corresponding $\hat{C}$ is $\begin{pmatrix}
1&\frac{-2}{(-z)^{i-1}}\\
0&(-1)^i\\
\end{pmatrix}$ and $\rho(\hat{C})=\{i-1,1-i\}$.
For $C=1$, $\hat{C}=1,\rho(\hat{C})=\{0\}$. For $C=z$, $\hat{C}=-1,\rho(\hat{C})=\{0\}$. 
By Cartan decomposition, $\rho(\hat{B})$ is unique. 
$B=g\Lambda h=g'\Lambda' h',\hat{B}=g\Lambda\tilde{\Lambda}\tilde{g}=g'\Lambda'\tilde{\Lambda}'\tilde{g}'$.
By abuse of notation, we write $z^i=\begin{pmatrix}
1&z\\
0&z^{i}\\
\end{pmatrix}$.
For $\Lambda=\oplus_{i}z^{c_i}, $ we denote $\sigma(\Lambda)=\{c_1,...c_j \}$. Here a block diagonal matrix $\Lambda$ has $z^{c_i}$ on the diagonal in the order of valuation non-decreasing.
We rewrite $\sigma(\Lambda)=\{a_1^{n_1},a_2^{n_2},...,1^{l_1},0^{r_1}\},  \sigma(\Lambda')=\{b_1^{n_1'},b_2^{n_2'},...,1^{l_2},0^{r_2}\}$. Therefore \[\rho(\hat{B})=\rho(\Lambda\tilde{\Lambda})= \{(a_1-1)^{n_1},(a_2-1)^{n_2},...,0^{l_1+r_1},(1-a_2)^{n_2},(1-a_1)^{n_1} \}=\rho(\Lambda'\tilde{\Lambda}').\] Thus we have these equalities: $a_i=b_i,n_i=n_i \text{ for all i, and } l_1+r_1=l_2+r_2.$ It suffices to show $l_1=l_2, r_1=r_2$ for this part.

We prove this part by contradiction.
Without loss of generality, $l_1>l_2$. 
We rewrite $ g\Lambda\tilde{\Lambda}\tilde{g}=g'\Lambda'\tilde{\Lambda}'\tilde{g}' \text{ as }g_0\Lambda\tilde{\Lambda}=\Lambda'\tilde{\Lambda}'g_0(-z)$. 

We write the $(i,j)$-th entry in the matrix $g_0$ in the form of  $a_{i,j}=a_{i,j,e}+a_{i,j,o}$, where $a_{i,j,o}$ is the odd part and $a_{i,j,e}$ is the even part.
By direct computation of the first $l_1+r_1$ by $l_1+r_1$  submatrix of LHS-RHS $(\text{i.e.,}g_0\Lambda\tilde{\Lambda}-\Lambda'\tilde{\Lambda}'g_0(-z))$, we have  

$\begin{pmatrix}
a_{1,1,o}&a_{1,2,o}&\cdots&a_{1,r_1+1,e}&\cdots&a_{1,l_1+r_1,e}\\
a_{2,1,o}&a_{2,2,o}&\cdots&a_{2,r_1+1,e}&\cdots&a_{2,l_1+r_1,e}\\
...\\
-a_{r_2+1,1,e}&-a_{r_2+1,2,e}&\cdots&-a_{r_2+1,r_1+1,o}&\cdots&-a_{r_2+1,l_1+r_1,o}\\
...\\
\end{pmatrix}=0$.

By direct computation on submatrix with rows  $\{l_1+r_1,l_1+r_1+1,...,n\}$ and columns  $\{1,2,...,l_1+r_1\}$  of 
 LHS-RHS,   we have
$a_{l_1+r_1+2k,l}= a_{l_1+r_1+2k,l,e} \text{ or } a_{l_1+r_1+2k,l,o}\text{ or }0$ and $v(a_{l_1+r_1+2k,l})\geq m_k$ for $l\leq l_1+r_1,k> 0$. 
By direct computation on submatrix with rows $\{1,2,...,l_1+r_1\}$ and columns  $\{l_1+r_1,l_1+r_1+1,...,n\}$  of LHS-RHS,  we have 
$a_{l,l_1+r_1+2k+1}= a_{l,l_1+r_1+2k-1,e} \text{ or } a_{l,l_1+r_1+2k-1,o}\text{ or }0$ and $v(a_{l,l_1+r_1+2k+1})\geq m_k-1$ for $l\leq l_1+r_1,k>0$. 
By direct computation on submatrix with rows $\{l_1+r_1,l_1+r_1+1,\ldots,n\}$ and columns  $\{l_1+r_1,l_1+r_1+1,...,n\}$  of LHS-RHS, we have  
$a_{l_1+r_1+2k',l_1+r_1+2k+1}= a_{l_1+r_1+2k',l_1+r_1+2k+1,e} \text{ or } a_{l_1+r_1+2k',l_1+r_1+2k+1,o}\text{ or }0$ where the valuation $\ge \text{ max } (m_k,m_{k'})-1$ for $k',k>0$. 
Here is an illustration of these rules, where $X$ represents an element complying to the restriction. (We only need the fact that each nonzero term has valuation greater than 0.)

$\begin{pmatrix}
a_{1,1,e}&\cdots&a_{1,l_1+r_1,o}&X&a_{1,l_1+r_1+2}&\cdots&X&a_{1,n}\\
...\\
a_{l_1+r_1,1,e}&\cdots&a_{l_1+r_1,l_1+r_1,o}&X&a_{l_1+r_1,l_1+r_1+2}&\cdots&X&a_{l_1+r_1,n}\\
a_{l_1+r_1+1,1}&\cdots&a_{l_1+r_1+1,l_1+r_1}&a_{l_1+r_1+1,l_1+r_1+1}&a_{l_1+r_1+1,l_1+r_1+2}&\cdots&a_{l_1+r_1+1,n-1} &a_{l_1+r_1+1,n}\\
X&\cdots&X&X&a_{l_1+r_1+2,l_1+r_1+2}&\cdots &X&a_{l_1+r_1+2,n}\\
...\\
a_{n-1,1}&\cdots&a_{n-1,l_1+r_1}&a_{n-1,l_1+r_1+1}&a_{n-1,l_1+r_1+2}&\cdots&a_{n-1,n-1} &a_{n-1,n}\\
X&\cdots&X&X&a_{n,l_1+r_1+2}&\cdots &X &a_{n,n}\\
\end{pmatrix}$. 

From Leibniz's formula of determinant and $g_0\in \text{GL}_n(k[[z]])$, the least valuation of nonzero summands is 0. We write one such summand with valuation 0 by $\tau$.  Notice that in row $l_1+r_1+2k \text{ } (k>0)$, nonzero $a_{l_1+r_1+2k,l} \text{ } a_{l_1+r_1+2k,l_1+r_1+2k'-1}\text{ where }(k'>0,l\leq l_1+r_1)$  have valuations greater than 0. Thus, we require that $\tau $ picks $\frac{n-l_1+r_1}{2}$ elements in these rows from columns $\{l_1+r_1+2,l_1+r_1+4,...,n\}$. In rows $\leq l_1+r_1$, no element in columns $\{l_1+r_1+2,l_1+r_1+4,...,n\}$ is in $\tau$. From observation on column $l_1+r_1+2k-1 (k>0)$, nonzero $a_{l_1+r_1+2l,l_1+r_1+2k+1}$ for all $l,k$ have valuations greater than 0.
Thus, we notice that all elements in $\tau $ which are in the first $l_1+r_1$ rows are in the first $l_1+r_1$ column. That is, in the first $l_1+r_1$ by $l_1+r_1$  submatrix of $g_0$ (denote by $g_1$), there exists a summand of Leibniz formula with a 0 valuation.

$g_1=\begin{pmatrix}
a_{1,1,e}&a_{1,2,e}&\cdots&a_{1,r_1+1,o}&\cdots&a_{1,l_1+r_1,o}\\
a_{2,1,e}&a_{2,2,e}&\cdots&a_{2,r_1+1,o}&\cdots&a_{2,l_1+r_1,o}\\
...\\
a_{r_2+1,1,o}&a_{r_2+1,2,o}&\cdots&a_{r_2+1,r_1+1,e}&\cdots&a_{r_2+1,l_1+r_1,e}\\
...\\
a_{l_1+r_1,1,o}&a_{l_1+r_1,2,o}&\cdots&a_{l_1+r_1,r_1+1,e}&\cdots&a_{l_1+r_1,l_1+r_1,e}\\
\end{pmatrix}$. 

For nonzero $a_{i,j,e}, a_{i,j,o}$, we have $v(a_{i,j,e})\geq 0, v(a_{i,j,o})\geq 1$. Any nonzero summand will have valuation $\geq |r_1-r_2|.$
The contradiction appears since $r_1<r_2$.
\end{proof}

We define the \emph{symmetric coweight} $\sigma$ on the blockwise decomposition. Recall in the proof of uniqueness in Theorem ~\ref{start_decomp}, for $\Lambda=\oplus_{i}z^{c_i}, $ we have $\sigma(\Lambda)=\{c_1,...c_j \}$. Therefore, for any $B\in \text{GL}_n (k[[z]])\Lambda \text{GL}_n (k((z^2)))$, we extend the definition of $\sigma(\Lambda)$ to $B$ such that $\sigma(B)=\sigma(\Lambda)$.

\subsection{Pieri and Dual Pieri Rule}

Having classified matrices of this form, it is natural to ask how a small change to the matrix affects the class of the matrix. We have already considered a version of this in the proof of Theorem~\ref{start_decomp}, where we added a single row and column to a matrix in block form and reduced it to block form. Another natural operation corresponds to a meromorphic change of basis in the difference equation, or in terms of $B$, left-multiplying it by a matrix $A$ which is not invertibly holomorphic. There are two cases in which we obtain particularly nice results on how the type can change.

In this subsection, we define two relations $a\sim b \text{ and } a\sim' b $ corresponding to  $a_i-1\le b_i\le a_i+1 \text{ and } a_{i+1}\le b_i\le a_{i-1}$ for all i. In the language of Young diagram, it is equivalent to: 
$\lambda\sim\mu$ means that $\lambda $ can be obtained from $\mu$ by adding and subtracting a vertical strip and $\lambda\sim'\mu$ means that $\lambda $ can be obtained from $\mu$ by adding and subtracting a horizontal strip.

\begin{theorem}
\label{pieri}
For any $A\in \text{GL}_n (k((z)))$ with  $\rho(A)=\{1^l\}$ and any B$\in  \text{GL}_n (k((z)))$ with $\sigma(B)=\{a_1,...,a_n\}$, we write $\sigma(AB)=\{b_1,...,b_n\}$.
Then, we have $\sigma(B)\sim \sigma(AB).$

\end{theorem}

\begin{theorem}
\label{dualpieri}
For any $A\in \text{GL}_n (k((z)))$ with  $\rho(A)=\{l\}$ and any B$\in  \text{GL}_n (\mathbb{k}((z)))$ with $\sigma(B)=\{a_1,...,a_n\}$, we write $\sigma(AB)=\{b_1,...,b_n\}$.
Then, we have $\sigma(B)\sim' \sigma(AB).$

\end{theorem}
We need two lemmas to prove the above two theorems. 
\begin{lemma}
\label{lemma1}
For any $A\in \text{GL}_n ({k}((z))), \text{ we write }\rho(A)=\{{a_1},\ldots, {a_n}\}$. If $a_n\geq 0$,  $\sigma(A)=\{{b_1},\ldots,{b_n}\}$ has the property $b_j\leq a_j$ for all $j$.
\end{lemma}

\begin{proof}

If $a_n\geq 0$, each element in $A$ has nonnegative valuation. Recall in Theorem~\ref{start_decomp}, in the algorithm of decomposing $gAh=\Lambda$, $h$ can be written as a  product of transvectional matrices with entries in $k[[z^2]]$, diagonal matrices with entries of nonpositive valuation and permutation matrices: In Case 2 Part 1 step 2, if $k'-1<a,$ then, instead of making $X_{3,2}\to 0$, we do similar column operations to make $X_{3,3}\to 0$. In Case 2 Part 2 condition $\mathbf{A}'$ and $\mathbf{B}'$, we do column multiplication on column 2,3 to make $X_{2,3},X_{1,4}\in \{1,z\}.$ Then, all transvectional matrix (column operations) have entries in $k[[z^2]]$.   Thus $\rho(H)=\{h_1,...h_n\}$ has  $h_1\leq 0$. $\sigma(A)=\sigma(\Lambda)=\rho(\Lambda)=\rho(GAH)=\rho(AH)$. From \cite[Chapter~5]{macdonald1998symmetric} , $b_j\leq a_j$ for all $j$.  
\end{proof}

\begin{lemma}
\label{lemma2}
For any $A\in \text{GL}_n (k((z))), \text{ we write }\rho(A)=\{{a_1},...{a_n}\}$. If  $a_{n-1}\geq 0 \text{ and }a_{n}<0$ , $\sigma(A)=\{{b_1},...{b_n}\}$ has the property $b_j\leq  a_{j-1}$ for all j.

\end{lemma}

\begin{proof}
Claim:
There exists a $D\in \text{GL}_n (k((z)))$, where the first row is $(z^{w},z^{c},0,..,0) $ ($w\in \{0,1\};c>w; w, c\text{ are of different parity} ) \text{ or } (z^w,0,0,..,0)$ and the first column is $(z^w,0,0,..,0)$, with the following properties: $\sigma(A)=\sigma(D)$; the lower right $n-1$ by $n-1$ submatrix of $D$ can be written in the form of  $D'[z^a]$, where $D'\in\text{GL}_{n-1} (k((z)))$, $\rho(D')=\{a_1,a_2,...a_{n-1}\}$, $[z^a]$ denotes a $n-1$ by $n-1$ diagonal matrix with the first entry $z^a$ and other entries 1 and $a\geq 0$; if $w=1$, then $a=0$.

We denote $\rho(D)=\{d_1,...d_{n-1},d_n\}$. With the claim, Lemma~\ref{lemma1} implies $b_i\leq d_i$ for all $i$.
Recall the original Dual Pieri rule for $\rho$ states: for $E \text{ and }E'=XE$, with $\rho(X)=\{f\}$, $f>0$, and $\rho(E)=\{e_1,..e_n\}$, $\rho(E')=\{f_1,..f_n\}$, we have $e_{i-1}\geq f_i\geq e_i, f_i\leq e_i+f$.  In other word, there is no vertical strip on the diagram. Thus, $\rho(D'[z^a])$=$\{\tilde{d_1},...\tilde{d_{n-1}}\}$ has  $\tilde{d_i}\leq a_{i-1}$ for $i<n$. One direct result from the claim is that $D$ can be decomposed in to $D'[a]$ and $z^w$. Thus, if $w=1$, $\{\tilde{d_1},...\tilde{d_{n-1}}\}=\rho(D'[a])=\rho(D')=\{a_1,a_2,...a_{n-1}\}$. Inserting one additional $1$ ($z^w$ part) leads to $d_i<a_{i-1}$ for all $i\leq n$. If $w=0$, we can rewrite $\rho(D)=\{\tilde{d_1},...\tilde{d_{n-1}},0\}$. Thus we have the property $d_i\leq a_{i-1}$ for all $i$ (if $i=n$, $d_n=0\leq a_{n-1}$).
Combining with $b_i\leq d_i$, the inequality $b_i\leq a_{i-1}$ is reached for all $i$.

Proof of the claim:

We will present an algorithm to find such $D$. All matrices operations are allowed operations from Theorem~\ref{start_decomp}.
Pick $A_{i,j}$ to be an entry in $A$ with the least valuation, i.e., $v(A_{i,j})=a_n$. Let  $B=g_0 (1,i) A (1,j)$ such that $B_{1,1}=z^{a_n}$, $B_{i,1}=0$ for $i>1$, where $(1,i),(1,j)$ are permutation matrices and $g_0\in GL_n(k[[z]])$. (We permute rows and columns such that $A_{i,j}$ is in the 1,1 position and reduce it to $z^{a_n}$ and eliminate all other entries in column 1.) The bottom right $n-1$ by $n-1$ submatrix has $\rho$=$\{a_1,..a_{n-1}\}$. Now, we have $v(B_{1,l})\geq a_n$ for $l>1.$ By trick 3 from Theorem~\ref{start_decomp}, $B_{1,l}=z^{c_l} f_{l}$ or $0$,  where $a_n< c_l$ and $\{a_n,c_l\}$ are in different parities and $f_l\in k[[z^2]]^*$ for all $l>1$. If $B_{1,i}$'s are all $0$, then the proof of the claim is done: $D_{1,2}=0,a=0$. If not all $B_{1,l}$ is $0$, we pick the least $c_l$ (assume it is $c_p$) and do trick 1 to remove $f_{l}$. Then we cancel all other $B_{1,l}$'s by column operations and switch column $p, 2$. Let $C$ denote the $g_1B h_1$ where $g_1\in GL_n(k[[z]]),h_1\in GL_n(k[[z^2]]),$ and $h_1$ is the above column operations, $g_1$ is the above row operation. Remember that the bottom right $n-1$ by $n-1$ submatrix of $C$ still has $\rho$=$\{a_1,..a_{n-1}\}$. We denote the submatrix $[C]_{n-1}$.
$C$ has the first row $(z^{a_n},z^{c_p},0,..,0)$. $c_p>a_n\text{ and }\{c,a_n\}$ are of different parity. 

If $c_p\geq 0$ and $2|a_n$,  $C$ right multiplies by $\mathbbm{1}_{(1,1)}(z^{-a_n})$. In this case, $w=0, c=c_p, a=0$.
If $c_p\geq 0$ and $2|a_n+1$ and $c_p\neq 0$, $ C$ right multiplies by $\mathbbm{1}_{(1,1)}(z^{1-a_n})$. In this case $w=1,c=c_p, a=0$.
If $c_p\geq 0$ and $2|a_n+1$ and $c_p=0$, $C$ right multiplies by $\mathbbm{1}_{(1,1)}(z^{1-a_n}) (1,2)$ and left multiplies by $\prod_{j>1} \mathbbm{1}_{(j,1)}(-C_{j,2})$. In this case, $w=0,c=1, a=1$.
If $c_p< 0$ and $2|a_n$,  $C$ right multiplies by 
$\mathbbm{1}_{(1,1)}(z^{-a_n}) \mathbbm{1}_{(2,2)}(z^{1-c_p})$. In this case $w=0, c=1, a=1-c_p$.
If $c_p< 0$ and $2|a_n+1$,  $C$ right multiplies by 
$\mathbbm{1}_{(1,1)}(z^{c_p+1-a_n}) (1,2) \mathbbm{1}_{(1,1)}(z^{-c_p})\mathbbm{1}_{(2,2)}(z^{-c_p})$ and left multiplies by $\prod_{j>1} \mathbbm{1}_{(j,1)}(-C_{j,2}z^{-c_p})$. In this case, $w=0, c=1, a=1-c_p$.
\end{proof}

\begin{proof}[Proof of Theorem \ref{pieri}]

We write $gABh=\Lambda, g'Bh'=\Lambda'$ with $g,g'\in \text{GL}_n(k[[z]])$ and $h,h'\in \text{GL}_n(k((z^2)))$. $\Lambda,\Lambda'$ are block diagonal matrices.

$\sigma(\Lambda')=\rho(\Lambda')=(a_1,a_2,...,a_n).$ $\sigma(\Lambda)=\sigma(Ag'^{-1}\Lambda')=(b_1,b_2,...,b_n)$. Also $\rho(Ag'^{-1}=\{1^l\}$. Therefore, $\rho(Ag'^{-1}\Lambda')=(c_1,c_2,..,c_n)$ has $a_i \leq c_i\leq a_i+1$ for all $i\leq n$ (original Pieri rule). By Lemma~\ref{lemma1},  $b_i\leq c_i$ for all i. Therefore, $b_i\leq a_i+1$.

For the other inequality, i.e., $a_i-1\leq b_i$ , we take $A'$ to be $A'=zA^{-1}g^{-1}$. Now, one direct result from A is $\rho(A')=\{1^{n-l}\}$. Notice that $A'gABh=A'\Lambda,A'gABh=zBh$. Thus, $\sigma(A'\Lambda)=\sigma(z\Lambda')=(\bar{a_1},\bar{a_2},...,\bar{a_n})$. Since multiplying $z$ by $\Lambda'$ only switches $1$ and $z$, we know $\bar{a_i}=a_i$ for $a_i>1$ and all other $\bar{a_i}$ are less than 2. Applying the result from the previous paragraph on $\Lambda$ and $A'\Lambda$ implies $\bar{a_i}\leq b_i+1$. This equality leads to $a_i\leq b_i+1$ from the fact that $b_i\geq 0.$ 
\end{proof}

\begin{proof}[Proof of Theorem \ref{dualpieri}]

We write $gABh\Lambda\text{ and } g'Bh'=\Lambda'$ with $g,g'\in \text{GL}_n(k[[z]])$ and $h,h'\in \text{GL}_n(k((z^2)))$,where $\Lambda,\Lambda'$ are block diagonal matrices.

We have $\sigma(\Lambda')=\rho(\Lambda')=(a_1,a_2,...,a_n)$ and  $\sigma(\Lambda)=\sigma(Ag'^{-1}\Lambda')=(b_1,b_2,...,b_n)$. Also $\rho(Ag'^{-1}=\{l\}$. Therefore, $\rho(Ag'^{-1}\Lambda')=(c_1,c_2,...,c_n)$ has $a_i \leq c_i\leq a_{i-1}$ for all $i\leq n$ (original Dual Pieri rule). By Lemma~\ref{lemma2},  $b_i\leq c_i$ for all i. Therefore, $b_i\leq a_{i-1}$.

For the other inequality, i.e., $a_i\leq b_{i-1}$, we take $A'$ to be $A'=A^{-1}g^{-1}$. Similarly to the proof of Theorem~\ref{pieri}, $\sigma(\Lambda')=\sigma(A'\Lambda)$. We write $\rho(A'\Lambda)=(e_1,...,e_n)$.
Recall the original dual Pieri rule for $\rho$ of $A',A'\Lambda$  shows $b_{i+1}\leq e_i\leq b_i$.
If $e_n\geq 0$, by Lemma~\ref{lemma1}, we have $a_i\leq e_i$. Thus $a_i\leq b_i\leq b_{i-1}$. If $e_n<0$, then by the inequality on original dual Pieri rule, we have $e_{n-1}\geq 0$. By Lemma~\ref{lemma2}, we have $a_i\leq e_{i-1}$. Thus, both cases yield $a_i\leq b_{i-1}$. 
\end{proof}

\section{Hecke Module}
In section2.2, Pieri and Dual Pieri rules provide two nice results on how left-multiplying $B$ by a matrix $A$ with certain type can change the symmetric coweight of $B$. The next natural question to ask is how spreading can the symmetric coweight be after the left multiplication. Similarly to the classical study of (spherical) Hecke algebra, we will define a Hecke module over double coset structure to describe the spreading with respect to Haar measure.

Throughout this chapter, we write $K=\text{GL}_n(k[[z]]),H=\text{GL}_n(k((z^2))), G=\text{GL}_n(k((z))).$ 
We define $\pi^\lambda=\text{diag}(z^{\lambda_n},...,z^{\lambda_1})$ and $c_\lambda$ is the characteristic function of the double coset $K\pi^\lambda K$. 
We define $\Pi^\lambda$ the block diagonal matrix with valuations in nondecreasing order. Let  $d_\lambda$ be the characteristic function of the double coset $K\Pi^\lambda H$.

Recall the construction of (spherical) Hecke algebra $\mathcal{H}(G,K)$ in \cite[Chapter~5]{macdonald1998symmetric},  MacDonald proves that $c_\lambda$'s form a $\mathbbm{C}$-basis of $\mathcal{H}(G,K)$. 

\begin{definition}
The Hecke module $\mathcal{H}(G,H,K)$ is the $\mathbbm{C}-$ vector space of all complex valued continuous functions on $G$  that is a linear combination of  $d_\lambda$. Equivalently, $d_\lambda$'s form a $\mathbbm{C}$-basis of $\mathcal{H}(G,H,K)$. We define a multiplication on $\mathcal{H}(G,H,K)$ and  $\mathcal{H}(G,K)$: $f\times g(x)=\int_G g(y^{-1}x)f(y)dy$ for $f\in \mathcal{H}(G,K)$, $g\in \mathcal{H}(G,H,K)$, $x\in G$. By Proposition~\ref{defmodule}, the Hecke module $\mathcal{H}(G,H,K)$ is a left $\mathcal{H}(G,K)$-module.
\end{definition}
Unlike the definition of the (spherical) Hecke algebra, in this case, $H$ is not compact, thus a definition of all continuous function with compact support cannot be applied here.
Any function in $\mathcal{H}(G,H,K)$ is invariant with respect to $(H,K)$, i.e., $f(kxh)=f(x)$ for all $x\in G$, $h\in H$, $k\in K$.

\begin{proposition}
\label{defmodule}
$\mathcal{H}(G,H,K)$ is a left $\mathcal{H}(G,K)$-module.
\end{proposition}

\begin{proof}
For all $ h\in H,k\in K$, 
\begin{align*}
    f\times g(kxh) &=\int_G g(y^{-1}kxh)f(y)dy=\int_G g(y^{-1}kx)f(y)dy\\&=\int_G g(y'^{-1}x)f(y')dy'=\int_G g(y^{-1}x)f(y)dy.
\end{align*}
Given $f_1,f_2\in \mathcal{H}(G,K) \text{ and } g\in \mathcal{H}(G,H,K)$, we have 
\begin{align*}
    f_1\times(f_2\times g)(x)&=\int_G (f_2\times g)(y^{-1}x)f_1(y)dy\\&=\int_G \int_G g(z^{-1}y^{-1}x)f_2(z)f_1(y)dydz\\&=\int_G \int_G g(y'^{-1}x)f_2(z'^{-1})f_1(y'z')dy'dz'.
\end{align*}
The last equality is obtained by  $y'=yz \text{ and } z'=z'^{-1}$.
Moreover, we have \[(f_1 f_2)\times g(x)=\int_G g(y^{-1}x)(f_1f_2)(y)dy=\int_G \int_G g(y^{-1}x)f_1(yz)f_2(z^{-1})dzdy.\] Thus, we have $f_1\times(f_2\times g)=(f_1 f_2)\times g$.

Similarly, for distributivity of $f$, we have
\begin{align*}
((f_1+f_2)\times g)(x) &=\int_G g(y^{-1}x)(f_1+f_2)(y)dy=\int_G (g(y^{-1}x)f_1(y)+g(y^{-1}x)f_2(y))dy\\
&=(f_1\times g+f_2\times g)(x).
\end{align*}
Thus, we have $(f_1+f_2)\times g=f_1\times g+f_2\times g.$

Similarly, for distributivity of $g$, given $g_1,g_2\in \mathcal{H}(G,H,K) \text{ and } f\in \mathcal{H}(G,K)$, we have 
\begin{align*}
f\times(g_1+g_2)(x) & =\int_G (g_1+g_2)(y^{-1}x)(f)(y)dy=\int_G (g_1(y^{-1}x)f(y)+g_2(y^{-1}x)f(y))dy \\
   & =(f\times g_`+f\times g_2)(x).
\end{align*}
Thus, we have $f\times(g_1+g_2)=f\times g_1+f\times g_2.$

Identity for $\mathcal{H}(G,K)$ is $c_0$, i.e., the characteristic function on $K$. Thus, \[c_0\times g(x)=\int_G g(y^{-1}x)c_0(y)dy=\int_K g(y^{-1}x)dy=g(x).\] 
\end{proof}

Given partitions $\lambda,\mu$ of length $n$, the product $c_{\mu}\times d_{\lambda}$ is a linear combination of $d_{\nu}$. Define $h_{\mu\nu}^{\lambda}$ by $c_{\mu}\times d_{\lambda}=\sum_{\nu} h_{\mu\nu}^{\lambda} d_{\nu}$. We have 
\begin{align*}
h_{\mu\nu}^{\lambda} &=(c_{\mu}\times d_{\lambda})(\Pi^{\nu}) \\
   &= \int_G d_{\lambda}(y^{-1}\Pi^\nu)c_{\mu}(y)dy \\
   & =\int_{K\pi^\mu K} d_{\lambda}(y^{-1}\Pi^\nu)dy\\
   &=\sum_{i}\int_K d_{\lambda}(k y_i^{-1}\Pi^\nu)dk\\
   &=\sum_{i}d_{\lambda}(y_i^{-1}\Pi^\nu) \\
\end{align*}
where $K\pi^{\mu}K=\cup_{i} y_{i}K$, and $K$ has measure 1. It follows that $d_{\lambda}(y_i^{-1}\Pi^\nu)=1$ if $y_i^{-1}\Pi^\nu\in K\Pi^\lambda H$ and  $d_{\lambda}(y_i^{-1}\Pi^\nu)=0$ if $y_i^{-1}\Pi^\nu\notin K\Pi^\lambda H$. 
Recall that the Haar measure of $K\pi^\mu K$ is finite; a direct result is that $h_{\mu\nu}^{\lambda}\in \mathbbm{Z}.$ 

In this chapter, we will first give a representative for $y_i$'s (Lemma~\ref{decomp_A}). Then,  $h_{\mu\nu}^{\lambda}$ is equal to the number of $ i$'s such that $y_i^{-1}\Pi^\nu\in K\Pi^\lambda H.$  Before we start Theorem~\ref{mainthm}, we introduce a system of \emph{paired tuple} and associated equivalence relation. The set $W_{\ell,\nu}^\lambda$ denotes the set of all paired tuples of size $\ell$ that are associated with the partitions $\nu$ and $\lambda$, which corresponds to symmetric coweights of elements in $\text{GL}_n(k((z)))$. Explicitly, we have $\ell(\nu)=n$, $m_0(\nu)\geq \sum_{i>1} m_i(\nu)$ and $m_i(\nu)=0$ if $i<0$, similarly for $\lambda$ as well.

Given partitions $\nu$ and $\lambda$ with $\ell(\nu)=\ell(\lambda)=n$, let $w$ be a tuple of length $n$ consisting of $0$'s and $1$'s. We say $w$ is a paired tuple of size $\ell$ associated with $\nu$ and $\lambda$ if the following conditions hold: 
\begin{enumerate}
\item  $\sum_{i}w_i=\ell$. Here $\ell$ is defined as the size of $w$.

\item Pairing is imposed on tuple $w$. There is at most one pairing for any index $i$.  If $\nu_i>1$, the indices $i,n-i+1$ form a pair $(i,n-1+1)$. For index $i$ with $\nu_i=1$, it can be paired with a index $j$ where $w_j=0$ and $\nu_j=0$, written as $(i,j).$ Let $\hat{\nu}=(\hat{\nu}_1,\dots,\hat{\nu}_{n})$ be a tuple obtained from $\nu$ by changing each part $\nu$ according to the pairing of $i$. We define:

$\hat{\nu}_i=
    \begin{cases}
      \nu_i+w_i-w_j, & \nu_i>1 \text{and } (i,j) \text{ is a pair}\\
      2, &\nu_i=1, \text{and } (i,j) \text{ is a pair} \\
      \nu_i+(-1)^{\nu_i}w_i, &   i \text{ is not paired}\\

      0, &\text{otherwise.}\\

    \end{cases}$
    
\item $\hat{\nu}$ is a reordering of $\lambda$.
\end{enumerate}

If $i>1$, denote $\omega_{[i]}=(\omega_{[i]}^{0,0},\omega_{[i]}^{1,1},\omega_{[i]}^{1,0},\omega_{[i]}^{0,1}),$ where \[ 
\omega_{[i]}^{a,b}= | \{ (r,s) : w_r = a, w_s = b, v_r = i \}|. 
\]
We also write $\omega_{[1]}=(\omega_{[1]}^{0,0},\omega_{[1]}^{1,1},\omega_{[1]}^{1,0},\omega_{[1]}^{0,1})$, where
\[
\omega_{[1]}^{a,b} = 
\begin{cases}
| \{\text{pair}(r,s):\nu_{r}=1 \}|, & \text{if } (a,b)=(0,0), \\
| \{r: w_r=\nu_{r}=1 \text{ and }r \text{ not paired} \} |, & \text{if } (a,b)=(1,1), \\
| \{r: w_r=1, \nu_{r}=0 \text{ and  }r \text{ not paired }\} |, & \text{if } (a,b)=(1,0), \\
0, & \text{if } (a,b)=(0,1).\\
\end{cases} 
\]
An equivalence relation is imposed on $w$ and $w'$:  $w\sim w'$ if and only if $\omega_{[i]}=\omega'_{[i]}$ for all $i$.

Next, we give a formula of $h_{\mu\nu}^{\lambda}$ for $\mu$ comprised of $-1$ and $0.$ 
For $a>1$, we denote $h_{-1^{2i+j+k},\{a^n\}}^{\{(a+1)^j,a^{n-j-k},(a-1)^k\}}\text{ as }  h_n^{{i,j,k}} $
and denote $\tilde{h}_a^{\omega_{[a]}}=h_n^{{i,j,k}} q^{-jk-(j+i)(n-i-k)}$  if $\omega_{[a]}=(n-i-j-k,i,k,j)$. We denote $h_{-1^{i+j+k},\{1^{n_2},0^{n_1}\}}^{\{2^i,1^{n_2-i-j+k\}}} \text{ as }\bar{h}_1^{\omega_{[1]}}$ where $\omega_{[1]}=(k,j,i,0).$

In Proposition~\ref{computen} and~\ref{compute1}, we will give a closed formula for $\tilde h,\bar h$. In the rest of this chapter, we write $t_1=n-2 \sum_{i>1}{m_i(\nu)}$.

\begin{theorem}
\label{mainthm}
\[h_{-1^\ell,\nu}^{\lambda}=\sum_{[w]\in W_{\ell,\nu}^{\lambda}/\sim} q^{r([w])} \bar{h}_1^{\omega_{[1]}}\prod_{\omega_{[i]}, i>1} \tilde{h}_i^{\omega_{[i]}}\]

\begin{multline*}
 r([w])=(\sum_{i>1,j\in \{0,1\}} \omega_{[i]}^{j,1})t_1+(\sum_{i>1}\omega_{[i]}^{0,0}-\omega_{[i]}^{1,1})(\sum_{j,k} \omega_{[1]}^{j,k})+\sum_{1<i<j} n_i n_j+ \sum_{1<i<j} n_i(\omega_{[j]}^{0,1}-\omega_{[j]}^{1,0})\\
 +\sum_{i>1} n_i \omega_{[i]}^{0,1}-\frac{1}{2} ((\sum_{i>1} \omega_{[i]}^{0,0}-\omega_{[i]}^{1,1})^2-(\sum_{i>1} \omega_{[i]}^{0,0} \omega_{[i]}^{0,0}+ \omega_{[i]}^{1,1}\omega_{[i]}^{1,1})). \\
\end{multline*}
\end{theorem}

Define $\mathcal{A}_a$ as the set that consists of all upper triangular matrices $X\in \text{GL}_n(k[[z]])$ with the following properties:
\begin{enumerate}
\item $X_{i,i}=1 \text{ or } z,$
\item for $i$ such that $X_{i,i}=z$,  we have $X_{i,j}=0$ if $j>i$ and $X_{j,i}\in k $ if $j<i$,
\item for i such that $X_{i,i}=1$, we have $X_{j,i}=0$ if $j<i $ and $X_{i,j}\in k $ if $j>i$,
\item the number of $z$ on the diagonal is $a$. 
\end{enumerate}

\begin{lemma}
\label{decomp_A}
Given $ A \in G$ with $\rho(A)=(1^a)$, we have $KAK=\bigsqcup_{X\in\mathcal{A}_a} KX$. 

\end{lemma}

\begin{proof}
Given $ X,Y\in\mathcal{A}_a,$ we first show that $KX\cap KY$ is empty unless $X=Y$. It suffices to show $XY^{-1}\notin K $. From direct computation, for $Y$ with entries $Y_{i,j}$, $Y^{-1}$ has ${Y^{-1}}_{i,j}=-Y_{i,j}/z$ if $i\neq j$ and ${Y^{-1}}_{i,i}=1/Y_{i,i}$. Since $X,Y$ are both upper triangular, on the diagonal of $XY^{-1}$, each entry is in $\{1,z,1/z\}$ and $XY^{-1}$ is also upper triangular. Thus, $X_{i,i}=Y_{i,i}$. Let $I = \{ i: X_{i,i}=z \}$.  For $j\in I,i\notin I,i<j$, we have \[{XY^{-1}}_{i,j}=X_{i,i}(-Y_{i,j}/z)+X_{i,j}(1/Y_{j,j})=-Y_{i,j}/z+X_{i,j}/z.\] Thus ${XY^{-1}}_{i,j}\in k[[z]] $ iff $X_{i,j}=Y_{i,j}.$ Recall that the Haar measure of $KAK$ is $ \binom{n}{a}$. Also, we have $\binom{n}{a}=|\mathcal{A}_a|$. Thus, $\mathcal{A}_a$ is a representative of the coset.
\end{proof}
Lemma~\ref{continuity} and~\ref{ktheory} are useful in the proof of Lemma~\ref{decomp_01},~\ref{decomp_flip},~\ref{weak}. 
\begin{lemma}
\label{continuity}
For any $A\in \text{GL}_n(k((z)))$ with $\rho(A)=\{\mu_1,...\mu_n\}$ where $\mu_1<a$ and $\mu_n\ge 0$, and  $A' \in \text{GL}_n(k((z)))$ with $A'-A\in M_n(z^ak[[z]])$, we have $\rho(A)=\rho(A')$.

\end{lemma}
\begin{proof}

We recall a classical result on determinant from \cite{marcus1990determinants} : \[\det(A+B)=\sum_r \sum_{\alpha, \beta\in I_r} (-1)^{s(\alpha)+s(\beta)} \det(A[\alpha|\beta])\det(B(\alpha|\beta))\]
Here $A$ and $B$ are $n$ by $n$ matrices, the outer sum on $r$ is over $\{0,...n\},$ $ I_r $ consists of all subsets in $\{1,...,n\}$ with size $r$, $A[\alpha|\beta]$ (square brackets) is the $r$ by $ r$ submatrix of $A$ in rows $\alpha$ and columns $\beta$, and $B(\alpha|\beta)$ (parentheses) is the $n-r$ by $n-r$ submatrix of $B$ in rows complementary to $\alpha$ and columns complementary to $\beta$. 

Denote $\rho_l=\sum_{i=n-l+1}^n \nu_i$ and $B=A-A'$. It suffices to show that gcd of $l$ by $l$ submatrix of $A+B$ is the same as that of $A$.  It suffices to prove the claim: for any $l<n$ and $\alpha,\beta\in I_l$, we have $v((A+B)[\alpha|\beta])\geq \rho_l$; pick $\alpha,\beta\in I_l $ such that $v(A[\alpha|\beta])=\rho_l$,$v((A+B)[\alpha|\beta])=\rho_l$. 

We first prove the second half of the claim. For the sake of simplifying the notation, denote $C=A[\alpha|\beta],D=B[\alpha|\beta]$. For $r<l \text{ and for all } \phi,\psi\in I_r$, nonzero summand  $(-1)^{s(\phi)+s(\psi)} \det(C[\phi|\psi])\det(D(\phi|\psi))$ has valuation $\geq(l-r)a+\rho_r\geq \rho_l.$ ($\nu_1<a,\rho_l-\rho_r\leq (l-r)a.)$ 
For $r=l \text{ and for all }  \phi,\psi\in I_r$, $(-1)^{s(\phi)+s(\psi)} \det(C[\phi|\psi])\det(D(\phi|\psi))$ has valuation equal to $\rho_l.$ For first half of the claim, the argument is the same as the $r<l$ part: For $r\leq l \text{ and for all }  \phi,\psi\in I_r$, nonzero summand  $(-1)^{s(\phi)+s(\psi)} \det(C[\phi|\psi])\det(D(\phi|\psi))$ has valuation $\geq(l-r)a+\rho_r\geq \rho_l$ (Note that $\nu_1<a,\rho_l-\rho_r\leq (l-r)a).$ 
\end{proof}
\begin{lemma}
\label{ktheory}
Denote the subgroup generated by transvection matrices, i.e., all $\mathbbm{1}_{(i,j)}(f)\text{ with }i\neq j$ and $v(f)\geq0$, as $E\subset K$. For any $g_0\in G$, there exists $k_1\in K$ and $e_2\in E$ such that $k_1g_0e_2=\Lambda$, where $\Lambda$ is a diagonal matrix of the form with diagonal entries $z^a$ and valuations of diagonal entries correspond to the dominant coweight of $g_0$.
\end{lemma}

\begin{proof}
Let $R$ be a commutative ring. By~\cite[Theorem~4.3.9]{hahn2013classical}, if $R$ is a Euclidean domain, then $E_n(R)=\text{SL}_n(R)$, where $E_n(R) $ is the subgroup of $\text{SL}_n(R)$ generated by transvections (also called elementary matrices).

Take $R=k[[z]]$. We notice that $\text{GL}_n(k[[z]])=(k[[z]]^*)\text{SL}_n(k[[z]])$. Cartan decomposition leads to $k_1g_0e_2=\Lambda$.
\end{proof}
In the proof of Lemma~\ref{decomp_01},~\ref{decomp_flip},  and~\ref{weak}, we fix matrix $B$ as follows: for $i$ with $\nu_i>1$, we have $B_{i,i}=1,B_{i,n-i+1}=z, B_{n-i+1,n-i+1}=z^{\nu_{i}}$; if $(n-t_1)/2 < i\leq m_0(\nu)$, then $B_{i,i}=1$; if $m_0(\nu)< i < m_0(\nu)+m_1(\nu)$, then $B_{i,i}=z$; other entries are $0$.  
Let $c_i$ denote column $i$ and $r_i$ denote row $ i$. Lemma~\ref{weak} is a weak version of Theorem~\ref{mainthm}.

\begin{lemma}
\label{decomp_01}
For any $X,X'\in\mathcal{A}_a$ satisfying $X_{i,j}=X'_{i,j}$ for $0<i,j\leq\frac{n-t_1}{2}$ or $\frac{n-t_1}{2}<i,j\leq\frac{n+t_1}{2}$  or  $\frac{n+t_1}{2}<i,j\leq n$,  we have $\sigma(XB)=\sigma(X'B)$.
\end{lemma}

\begin{proof}
It suffices to prove that the lemma holds if $X'$ satisfies the extra condition: $X'_{i,j}=0$ for ($\frac{n-t_1}{2}\geq i$ and $j>\frac{n-t_1}{2}$) or ($\frac{n-t_1}{2}<i\leq\frac{n+t_1}{2}$ and $j>\frac{n+t_1}{2}$).
We write $Y=XB$. We will do allowed row/column operations on $Y$ to prove this lemma.
This is a 6-step matrix operations procedure:

Step 1: For any $i\leq\frac{n-t_1}{2}$, there exist allowed operations such that $X_{i,n-i+1}$ can be reduced to $0$.

For $j\leq\frac{n-t_1}{2}$ and $i\neq j$, we do column operation to make $Y_{i,j}=0$. Thus, the first $\frac{n-t_1}{2}$ by $\frac{n-t_1}{2}$ submatrix of $Y$ is diagonal.
If $X_{i,i}=1$, $Y_{i,n-i+1}=z+X_{i,n-i+1}B_{n-i+1,n-i+1}$ can be reduced to $z$: It suffices to consider $X_{i,n-i+1}\neq 0,$  which implies $X_{n-i+1,n-i+1}=z.$ We write $B_{n-i+1,n-i+1}$ as $z^a$. If $2|a$, then $c_{n-i+1}$ replaced by $c_{n-i+1}-X_{i,n-i+1}B_{n-i+1,n-i+1}c_{i}$ will make $Y_{i,n-i+1}$ to $z$. If $2|a+1$, we replace $c_{n-i+1}$ by $c_{n-i+1}/(1+X_{i,n-i+1}B_{n-i+1,n-i+1}/z)$. Recall that on row $n-i+1$ the only nonzero term in $Y$ is $Y_{n-i+1,n-i+1}$, thus by row operations from row $n-i+1$ to all other rows, $Y_{l,n-i+1}/(1+X_{i,n-i+1}B_{n-i+1,n-i+1}/z)\to Y_{l,n-i+1}$ for all $l\neq n-i+1$. If $X_{i,i}=z$, $Y_{i,n-i+1}=z^2+X_{i,n-i+1}B_{n-i+1,n-i+1}$ can be reduced to $z^2$: similarly as $X_{i,i}=1$.

Step 2: For any $i\leq\frac{n-t_1}{2}$ and $\frac{n-t_1}{2}<j\leq\frac{n+t_1}{2}$ , there exist allowed operations such that $X_{i,j}$ can be reduced to $0$.

It suffices to consider $X_{i,j}\neq 0,$ which implies $X_{j,j}=z$,  $X_{i,i}=1.$  If $B_{j,j}=1$, $c_j$ can be replaced by $c_j-X_{i,j}c_i$. If $B_{j,j}=z$, $c_j$ can be replaced by $c_j-X_{i,j}c_{n-i+1}$. Notice that on column $n-i+1$, $Y_{i,n-i+1}=z$ and all other entries in $Y$'s column $n-i+1$ are 0 or of valuation $z^a,z^{a+1}$ if we write $B_{n-i+1,n-i+1}=z^a$.  On row $j$ the only nonzero term in $Y$ is $Y_{j,j}=z^2$. Thus,  by row operations from row $j$ to all other rows, $Y_{l,j}-X_{i,j}Y_{l,n-i+1}\to Y_{l,j}$ for all $l\neq i$.

Step 3: For any $i,j\leq\frac{n-t_1}{2}$ and $X_{i,i}=z$  and $i\neq j$ , there exist allowed operations such that $X_{j,n-i+1}$ can be reduced to 0, i.e., $Y_{j,n-i+1}=X_{j,i}B_{i,n-i+1}+X_{j,n-i+1}B_{n-i+1,n-i+1}$ reduces to $X_{j,i}z$. 

It suffices to consider $X_{j,n-i+1}\neq 0,$ which implies $X_{n-i+1,n-i+1}=z, X_{j,j}=1.$ The only two nonzero elements in $Y$'s row i are $Y_{i,i}=z,Y_{i,n-i+1}=z^2$. We write $B_{n-i+1,n-i+1}$ be $z^{a}$. If $2|a$, then $c_{n-i+1}$ replaced by $c_{n-i+1}-X_{j,n-i+1}B_{n-i+1,n-i+1}c_{j}$ will make $Y_{j,n-i+1}$ to $X_{j,i}z$. If $2| 1+a$, we replace $r_{j}$ by $r_j-X_{j,n-i+1}B_{n-i+1,n-i+1}/z^2$ ($B_{n-i+1,n-i+1}$ has valuation $\geq 2.$). Then a column operation $c_i=c_i+c_jX_{j,n-i+1}B_{n-i+1,n-i+1}/z^2$ will keep all other terms the same as before.

Step 4: For any $i\leq\frac{n-t_1}{2}$, $j>\frac{n+t_1}{2}, X_{n-j+1,n-j+1}=1 \text{ and } i<n-j+1$  , there exist allowed operations such that $X_{i,j}$ can be reduced to 0. 

Starting from $j=\frac{n+t_1}{2}\to n$, for all eligible $i$ satisfying the above condition, the following algorithm would remove $X_{i,j}.$ We write $B_{j,j}$ be $z^{a}$.
If $2|a$, then $c_{j}$ replaced by $c_{j}-X_{i,j}B_{j,j}c_{i}$ will make $Y_{i,j}$ to $0$.
If $2|a+1$, we replace $r_{i}$ by $r_i-r_{n-j+1}X_{i,j}B_{j,j}/z$( $B_{j,j}$ has valuation $\geq 2$). 
Recall on row $n-j+1$ of $Y$, $Y_{n-j+1,n-j+1}=1,Y_{n-j+1,j}=z;$ if $X_{n-l+1,n-l+1}=z, \frac{n+t_1}{2}<l<j$ ($v(Y_{n-j+1,l})=1 \text{ or } \infty)$, then $Y_{n-j+1,l}=X_{n-j+1,n-l+1}B_{n-l+1,l}$ ; if $X_{l,l}=z, l>j$. ($v(Y_{n-j+1,l})\geq 2 \text{ or } \infty)$, then $ Y_{n-j+1,l}=X_{n-j+1,l}B_{l,l}$.
Then we need to do the next three steps to make all other entries in row $i $ remaining the same.
\begin{enumerate}
    \item $ A$ column operation $c_{n-j+1}=c_{n-j+1}+c_i X_{i,j}B_{j,j}/z$. 
    \item If $X_{n-l+1,n-l+1}=z, \frac{n+t_1}{2}<l<j$, we redo step 3 to eliminate terms in $Y_{i,l}$ that is caused by the $r_{n-j+1}$ row operation.
    \item If $X_{l,l}=z$, $l>j$, we do row operation $r_i+r_l X_{i,j}B_{j,j}Y_{i,l}/(zY_{l,l})$ (row $l$ has only one nonzero term $Y_{l,l}$). 
\end{enumerate}
Notice that for step 4, we on purpose make current $Y$ and original $Y$ only differ by one element for each $i,j$. 

Step 5: For any $i\leq\frac{n-t_1}{2}$, $j>\frac{n+t_1}{2}, X_{n-j+1,n-j+1}=1 \text{ and } i>n-j+1$, there exist allowed operations such that $X_{i,j}$ can be reduced to 0. 

Starting from $j=n\to \frac{n+t_1}{2}$, for all eligible $i$ satisfying above condition, the following algorithm would remove $X_{i,j}$. We write $B_{j,j}$ be $z^{a}$.
If $2|a$, then $c_{j}$ replaced by $c_{j}-X_{i,j}B_{j,j}c_{i}$ will make $Y_{i,j}$ to $0$.
If $2|a+1$, we replace $r_{i}$ by $r_i-r_{n-j+1}X_{i,j}B_{j,j}/z$( $B_{j,j}$ has valuation $\geq 2$).
Recall that on row $n-j+1 $of $Y$, $Y_{n-j+1,n-j+1}=1,Y_{n-j+1,j}=z;$ if $X_{n-l+1,n-l+1}=z, \frac{n+t_1}{2}<l<j$ ($v(Y_{n-j+1,l})=1\text{ or } \infty)$, then $Y_{n-j+1,l}=X_{n-j+1,n-l+1}B_{n-l+1,l}$ ; if $X_{l,l}=z, l>j$. ($v(Y_{n-j+1,l})\geq 2 \text{ or } \infty)$, then $ Y_{n-j+1,l}=X_{n-j+1,l}B_{l,l}$.
Then we need to do the next two steps to make all other entries in row $i$ remain the same. 
\begin{enumerate}
    \item A column operation $c_{n-j+1}=c_{n-j+1}+c_i X_{i,j}B_{j,j}/z$. 
    \item If $X_{n-l+1,n-l+1}=z, \frac{n+t_1}{2}<l<j$, we redo step 3 to eliminate terms in $Y_{i,l}$ that is caused by the $r_{n-j+1}$ row operation.
\end{enumerate}
Notice that for step 5, we on purpose make current $Y$ and original $Y$ only differ by one element for each $i,j$. 

Step 6: For $n\geq j>\frac{n+t_1}{2}, \frac{n-t_1}{2}<i\leq\frac{n+t_1}{2} $, there exist allowed operations such that $X_{i,j}$ can be reduced to 0. 

It suffices to consider $X_{i,j}\neq 0,$ which implies $X_{j,j}=z, X_{i,i}=1.$  If $B_{j,j} \text{and } B_{i,i} $ are of the same parity, $c_j$ can be replaced by $c_j-X_{i,j}B_{j,j}/B_{i,i}c_i$. If $B_{j,j} \text{and } B_{i,i} $ are of different parities: if $X_{n-j+1,n-j+1}=z$,  apply step 3 again; if $X_{n-j+1,n-j+1}=1,$ apply step 5 again.
\end{proof}

In the following lemma,  $[M]_{n/2}$ denotes the $n/2,n/2$ submatrix of $M$ that consists of the  $\{n/2+1,n/2+2,...n\}$th rows and columns; $\bar{[M]}_{n/2}$ denotes the $n/2,n/2$ submatrix of $M$ that consists of the  $\{1,2,...n/2\}$th rows and columns. Let $\hat{[M]}_{n/2}$ denote a matrix $Y$ in $\text{GL}_n([[z]])$ such that $Y_{i,i}=1,Y_{i,n-i+1}=z$ for $i\leq n/2$ and $[Y]_{n/2}=[M]_{n/2}$ and all other entries are 0.
\begin{lemma}
\label{decomp_flip}
In this lemma, we impose a further condition on $B$: $t_1=0$. For any $X\in \mathcal{A}_a$, there exist $M,N\in \text{GL}_{n/2}(k((z)))$ such that: $N$ is diagonal with $N_{i,i}=1/X_{n/2-i+1,n/2-i+1}$; $X_{i,j}=-M_{n/2+1-i,n/2+1-j}$ for $0<i,j\leq n/2\text{ and }i\neq j$; $M_{i,i}=1$ for all $i$. Thus $\sigma(XB)=\rho([X]_{n/2} [B]_{n/2} M N)$. 
\end{lemma}

\begin{proof}
Lemma~\ref{decomp_01} shows that if $X,X'\in \mathcal{A}_a$ and they follow the restrictions in Lemma~\ref{decomp_01}, there exist $ k_0\in K,h_0\in H\text{ such that }k_0XBh_0=X'B$. It suffices to show: 1. There exist columns and rows operations on $XB$ to a matrix $Y$, where  $Y_{i,i}=1,Y_{i,n-i+1}=z$ for $i\leq n/2$ and $[Y]_{n/2}=[X]_{n/2} [B]_{n/2} M N$ and all other entries are 0. 2. On $[Y]_{n/2}$, given any transvectional matrix $e_3$ in $\text{GL}_{n/2}(k[[z]])$, there exist row and column operations such that $\hat{[Y]}_{n/2}$ turns to $\hat{[Y]_{n/2} e_3}$. Thus by Lemma ~\ref{ktheory}, $\sigma(XB)=\rho([X]_{n/2} [B]_{n/2} M N).$

For 1:
we add steps 7 and 8 after the former 6 steps for the sake of proof continuity. Step 7:  For $i\leq n/2, n/2< j<n-i \text{ satisfying  }Y_{i,i}=1\text{ and }X_{n-j+1,n-j+1}=z$ (recall $Y_{i,j}=X_{i,n-j+1}z$), column operation on $c_j$ by rewriting $c_j$ to be $c_j-c_{n-i+1}X_{i,n-j+1}$  makes $Y_{i,j}=0$. Then a first half of step 1 will make $Y_{i,n-j+1}$ 0. Now, $[Y]_{n/2}=[X]_{n/2} [B]_{n/2} M$.

Step 8. For all $i\leq n/2\text{ with }Y_{i,i}=z\text{ and for all }n/2<j<n$, row operation on $r_j$ by rewriting $r_j$ to be $r_j-r_iY_{j,n-i+1}/z^2$ makes  $Y_{j,i}=-Y_{j,n-i+1}/z$( $Y_{j,n-i+1}$ refers to the entry of $Y$ before step 8) and $Y_{j,n-i+1}=0$. After each $j$ is proceeded, we do a column switch for $c_i,c_{n-i+1}$. Then rewrite  $c_{n-i+1} \text{ as } -c_{n-i+1}; r_{i}\text{ as }-r_{i};c_{i}\text{ as } -c_{i}/z^2 $.  Notice that these procedures for $i$ are equivalent to rewrite  $c_{n-i+1} \text{ as }c_{n-i+1}/z $.

For 2, we write $e_3$ as a column operation that takes $c_i \text{ to be }c_i-fc_j $. This is equivalent, on $Y$, to $c_{n/2+i}-fc_{n/2+j}$ on bottom right submatrix of $Y$ while preserving other parts. Let $f_{o}$ denote the odd part of f and $f_e$ even part of f.  Similarly to step 8, for $n/2<l<n$, row operations on $r_l$ by writing $r_l$ to be $r_{l}-r_{n/2-j+1}Y_{l,n/2+j}/z$ will make $Y_{l,n/2+j}=-Y_{l,n/2-j+1}/z$ ( $Y_{j,n/2-j+1}$ refers to the entry of $Y$ before) and  $Y_{l,n/2-j+1}=0$. After each $l$ is proceeded, we do a column operation on $c_{n/2+i}$ by writing $c_{n/2+i}$ to be  $c_{n/2+i}-f_oc_{n/2-j+1}z$.  Notice that currently $Y_{n/2-j+1,i+n/2}=-f_oz$, thus  $r_{n/2-j+1}+r_{n/2-i+1}f_o\text{ as } r_{n/2-j+1} \text{ and } c_{n/2-i+1}-c_{n/2+j}f_o \text{ as }c_{n/2-i+1}$ will remove the term $Y_{n/2-j+1,i+n/2}$. Similarly to step 8, for $n/2<l<n$, row operations $r_{l}-r_{n/2-j+1}Y_{l,n/2-j+1}$ on $r_l$ will make $Y_{l,n/2+j}$ to be the original $Y_{l,n/2+j}$. For $f_e,$ column operation on $c_{n/2+i}$ by writing $c_{n/2+i}$ to be $c_{n/2+i}+f_ec_{n/2+j}$ makes column $c_{n/2+i}$ as $c_{n/2+i}-fc_{n/2+j}$. Notice currently $Y_{n/2-j+1,i+n/2}=f_ez$, row operation on $r_{n/2-j+1}$ by writing $r_{n/2-j+1}$ to be $r_{n/2-j+1}-r_{n/2-i+1}f_e$ and column operation on $c_{n/2-i+1}$ by writing $c_{n/2-i+1}$ to be $c_{n/2-i+1}+c_{n/2-j+1}f_e$ will remove the term $Y_{n/2-j+1,n/2-i+1}$.
\end{proof}
We first prove a weak version of Theorem~\ref{mainthm} with the help of the above lemmas:

\begin{lemma}
\label{weak}
Theorem~\ref{mainthm} holds if distinct parts of $\nu$ with $\nu_i>1$ differ by at least 5.
\end{lemma}

\begin{proof}
Define $I_{j}=[\sum_{i=2}^{j-1} m_i(\nu) +1,\sum_{i=2}^{j} m_i(\nu) ]$. Let $[C]_{I}$ denote the submatrix of C with rows and columns in $I$. Let $[C]_{I,J}$ denote the submatrix of C with rows in $I$ and columns in $J$.   

Claim: Following lemma~\ref{decomp_flip}, there exist $ g_1\in \text{GL}_{n/2}(k[[z]]), e_2\in E_{n/2} $ such that \[g_1 [X]_{n/2} [B]_{n/2} M N e_2=\bigoplus_i  z^{\rho([[X]_{n/2}]_{I_i}[[B]_{n/2}]_{I_i} [M]_{n/2-I_i} [N]_{n/2-I_i})}.\] In other word, if $i_1\neq i_2$, all terms in $[[X]_{n/2}]_{I_{i_1},I_{i_2}},[\bar{[X]}_{n/2}]_{I_{i_1},I_{i_2}}$ do not have any influence on $\sigma(XB)$. 

We denote $I'_1=[\frac{n-t_1}{2},\frac{n+t_1}{2}]$, and  \[I'_j=[\sum_{i=2}^{j-1} m_i(\nu) +1,\sum_{i=2}^{j} m_i(\nu) ]\cup (\frac{n+t_1}{2}+[\sum_{i=2}^{j-1} m_i(\nu) +1,\sum_{i=2}^{j} m_i(\nu)]).\] By lemma~\ref{decomp_01} and the claim, we have shown  $X,X'\in \mathcal{A}_\ell$, $\sigma(XB)=\sigma(X'B)$ if $[X]_{I'_i}=[X']_{I'_i}$  for all $I_i'$. In the proof of the claim,  we will show: \[\rho([[X]_{n/2}]_{I_i}[[B]_{n/2}]_{I_i} [M]_{n/2-I_i} [N]_{n/2-I_i})=\{(i+1)^a,i^b,(i-1)^c\}.\] We write the number of z and 1/z in $[[X]_{n/2}]_{I_i}\cup [N]_{n/2-I_i}$  to be d. Thus, $d-(a+c)\geq 0 \text{ and } 2|d-(a+c)$. In  the proof of Proposition ~\ref{compute1} , we will show on $I'_1$, $\sigma=\{2^a,1^b,0^c\}$ and same rule on a,b,c hold, with additional dimension requirement: $z^2$ exists only if 1 exists in $\sigma(B)$. These are the rules in the summation of the original theorem under $[w]$. 

The remaining part of this proof is to compute $r[w]$. 
In Lemma~\ref{decomp_01}, Step 2: For any $i\leq\frac{n-t_1}{2}$ and $\frac{n-t_1}{2}<j\leq\frac{n+t_1}{2}$ , there exist allowed operations such that $X_{i,j}$ can be reduced to 0. The number of all entries satisfying above condition in $X$'s is $(\sum_{i>1,j\in \{0,1\}} \omega_{[i]}^{0,j})(\sum_{i,j}\omega_{[1]}^{i,j})$.
In Lemma~\ref{decomp_01}, Step 6: for $n\geq j>\frac{n+t_1}{2}, \frac{n-t_1}{2}<i\leq\frac{n+t_1}{2} $, there exist allowed operations such that $X_{i,j}$ can be reduced to 0. 
The number of all entries satisfying the above condition in $X$'s is $(\sum_{i>1,j\in \{0,1\}} \omega_{[i]}^{j,1})(t_1-\sum_{i,j}\omega_{[1]}^{i,j})$.
Combining all other steps in Lemma~\ref{decomp_01}, for any $i\leq\frac{n-t_1}{2}$ and $n\geq j>\frac{n+t_1}{2}$ , there exist allowed operations such that $X_{i,j}$ can be reduced to 0. The number of all entries satisfying the above condition in $X$'s is $(\sum_{i>1,j\in \{0,1\}} \omega_{[i]}^{j,1})(\sum_{i>1,j\in \{0,1\}} \omega_{[i]}^{0,j})$.
From the claim, the number of all entries satisfying the condition in $X$'s is 
\[\sum_{1<j<i,}(\sum_{k\in \{0,1\}}\omega_{[i]}^{0,k})(\sum_{k\in \{0,1\}}\omega_{[j]}^{1,k})+\sum_{1<i<j,}(\sum_{k\in \{0,1\}}\omega_{[i]}^{k,0})(\sum_{k\in \{0,1\}}\omega_{[j]}^{k,1}).\]
From the definition of $\tilde{h}_a^{\omega_{[a]}}$ and Proposition~\ref{computen}, $\sum_{i>1} (\omega_{[i]}^{1,0}\omega_{[i]}^{0,1})$ need to be added in $r[w]$.

Therefore, \[r([w])=\sum_{i>1} \omega_{[i]}^{1,0}\omega_{[i]}^{0,1}+\sum_{1<j<i}(\sum_{k\in \{0,1\}}\omega_{[i]}^{0,k}\sum_{k\in \{0,1\}}\omega_{[j]}^{1,k}+\sum_{k\in \{0,1\}}\omega_{[j]}^{k,0}\sum_{k\in \{0,1\}}\omega_{[i]}^{k,1})\]\[+\sum_{i>1,j\in \{0,1\}} \omega_{[i]}^{j,1}\sum_{i>1,j\in \{0,1\}} \omega_{[i]}^{0,j}+\sum_{i>1,j\in \{0,1\}} \omega_{[i]}^{0,j}\sum_{i,j}\omega_{[1]}^{i,j}+(\sum_{i>1,j\in \{0,1\}} \omega_{[i]}^{j,1})(t_1-\sum_{i,j}\omega_{[1]}^{i,j}).\]
Up to reordering and simplifying, we have the $r[w]$ in the statement.

Proof of claim: 
We are still following our steps in Lemma~\ref{decomp_flip}. However, in the following proof, we will replace $[X]_{n/2} [B]_{n/2} M N$ by $Y$ temporarily for the simplicity of notation. The row and column operations refers to the transvectional matrix in $ E_{n/2}$ multiplying on the left and right, plus simple row multiplication by $f \text{ with }v(f)=0$.

We start from the first nonempty set $I_{i_1}$.  Denote $C_{i_1}= [[X]_{n/2}]_{I_{i_1}}[[B]_{n/2}]_{I_{i_1}} [M]_{n/2-I_{i_1}} [N]_{n/2-I_{i_1}}$. We will prove three statements: 1. $\rho(C_{i_1})$ consists of $\{i_1,i_1-1,i_1+1\}$;
2. There exist column and row operations that make $[Y]_{I_{i_1}}$ diagonal and remove all other elements in the same rows or columns with statement 3 always true;
3. Each term in $[Y]_{I_{i_j}}-C_{i_j}$ has valuation $>i_j+1$.  Each nonzero term in $[Y]_{I_{i_j},I_{i_l}}$ has valuation $\geq \text{max} (i_j,i_l)-1$. Therefore, employing 3 statements on $i$'s from all nonempty sets $I_{i}$ with $i$ in increasing order  will finish the proof of the claim. 

Statement 1: Recall that original Pieri rule guarantees that $\rho([[X]_{n/2}]_{I_{i_1}}[[B]_{n/2}]_{I_{i_1}})$ consists of $\{i+1,i\}$ with the number of i+1 being the number of $z$'s on the diagonal of  $[[X]_{n/2}]_{I_{i_1}}$. By the construction of $M$, all entries in $M $ are in $ k \text{ and } det(M)\neq 0$. Thus,  $\rho(C_{i_1})$ consists of $i_1-1,i_1,i_1+1$. and there exists $p\geq 0$ such that  the number of $i_1-1,i_1+1$ plus $2p$ is the number of $z$ or $1/z$ on the diagonal of $[[X]_{n/2}]_{I_{i_1}},[N]_{n/2-I_{i_1}}$.

Statement 2: Currently, each term in $[Y]_{I_{i_1}}-C_{i_1}$ has valuation $>i+1$ (directly from multiplication). Then, by Lemma~\ref{decomp_01}, $\rho([Y]_{I_{i_1}})=\rho(C_{i_1})$. By Lemma~\ref{decomp_flip}, the submatrix $[Y]_{I_{i_1}}$ can be diagonalized, with rows/columns operations. Then we remove all nonzero terms in column $I_{i_1}$ under the current block. On rows in $I_{i_j},$ the difference in elements before and after row operation will have valuation  $\geq i_j-1-(i_1+1)$. Thus on the $I_{i_j}$ diagonal block, the difference in any element has valuation $\geq 2(i_j-1)-(i_1+1)$. On the off diagonal block, $[Y]_{I_{i_j},I_{i_l}}$, the difference in any element has valuation $\geq i_l-1+i_j-1-(i_1+1)$. Then we  remove all terms in rows $I_{i_1}$ after the $I_{i_1}$ block. Notice that with the gap of greater than 4, all above inequality satisfies the statement 3. We redo this procedure to make $Y$ blockwise diagonal. 
\end{proof}

Before we close the gap of Theorem~\ref{mainthm}, we will first provide closed formulas for $h$ with restrictions on $\nu$. Since the proof of closing the gap is pure combinatorial and lengthy, we will give a self-contained proof for the following two theorems (Proposition ~\ref{computen} is a simple case of Proposition~\ref{young}). In the rest of this chapter, we will define B slightly different than before. Let $B$ denote the block diagonal matrix with valuations in nondecreasing order and $\sigma(B)=\nu$. Recall from the beginning of this chapter, this is the definition of $\Pi^{\nu}$.

In this theorem, we have $\nu=\{a^n\}$. Equivalently, on $B$,  it consists of $z^a$ blocks.
\begin{proposition}

\label{computen}
\[\tilde{h}_a^{(n-i-j-k,i,k,j)}={\binom{n}{i,j,k}}\text{ for all a with } a>1.\]
\end{proposition}

\begin{proof}
We will prove $h_{-1^{2i+j+k},\{a^n\}}^{\{(a+1)^j,a^{n-j-k},(a-1)^k\}}=q^{n(i+j)-i(i+j+k)}{\binom{n}{i,j,k}}$.

Let $h_n^{i,j,k} \text{ denote }h_{-1^{2i+j+k},\{a^n\}}^{\{(a+1)^j,a^{n-j-k},(a-1)^k\}}$. We will give a detailed proof based on the matrix decomposition of Proposition~\ref{young}. However, instead of employing the entire paired semi-tableau construction, we only borrow the matrix decomposition techniques and the statement:  All elements that must be $0$s (originated from the rule in $X$) still remain to be $0$ under some row operations.  We prove by induction on $n$, where $n$ is the number of two dimensional blocks. $i,j, k$ correspond to $\pm1, +1,-1$. That is, for $B$ consists of $n$ $z^a$ block, $h_n^{{i,j,k}}$ is the number of $X$ in $\mathcal{A}_{2i+j+k}$ such that $\sigma(XB)=\{(a+1)^{j},a^{n-j-k},(a-1)^{k}\}$. 
Assume the statement holds for all $h_{n'}^{i,j,k}$ with $n'<n$. 
We will prove the following equality:
\begin{align*}
h_n^{{i,j,k}} & = h_{n-1}^{{i,j,k}}+h_{n-1}^{{i,j,k-1}} q^{n-k-i}+h_{n-1}^{{i,j-1,k}} (q^{2n-k-j-2i}+q^{n-i-1}(q^{n-i-j-k}-1)) \\
   &+h_{n-1}^{{i-1,j,k+1}}q^{n-i-k-1}(q^{k+1}-1)+h_{n-1}^{{i-1,j,k}}q^{2n-2i-k}+h_{n-1}^{{i,j-1,k-1}}q^{2n-2i-k-1}(q^{n-i-j-k+1}-1) \\
   &+h_{n-1}^{{i,j-2,k}}q^{3n-3i-j-k-1}(q^{n-i-j-k+1}-1)+h_{n-1}^{{i-1,j-1,k+1}}q^{3n-3i-2k-j-1}(q^{k+1}-1).
\end{align*}
Following the general proof of Proposition~\ref{young}, we embed a $(2n-2)\text{ by } (2n-2)$ $X'\in \mathcal{A}$ into  $X$ $\in \mathcal{A}_{2i+j+k}$ as the first $(2n-2)\text{ by } (2n-2)$ submatrix. Denote $\sigma(X'[B])=\{(a+1)^{j'},a^{n-j'-k'},(a-1)^{k'}\}$. (By abusing the notation, we write $[B]$ to be the first $(2n-2)\text{ by } (2n-2)$ submatrix of $B$.) Then on columns $2n-1,2n$ of $X$, the following changes in $\sigma$ can be achieved: 
\begin{enumerate}
\item If there is no $z$'s in $X_{2n-1,2n-1}, X_{2n,2n}, \text{ then }\sigma(XB)=\{(a+1)^{j'}, a^{n-j'-k'}, (a-1)^{k'}\}$ (summand 1).  

\item If there is $1$ $z$'s in $X_{2n-1,2n-1}, X_{2n,2n}, \text{ then } \sigma(XB)= \{(a+1)^{j'}, a^{n-j'-k'-1}, (a-1)^{k'+1}\} \text{ (summand 2)}, \{(a+1)^{j'+1}, a^{n-j'-k'-1}, (a-1)^{k'}\}\text{ (summand 3)}$, $\{(a+1)^{j'}, a^{n-j'-k'+1}, (a-1)^{k'-1}\}\text{ (summand 4)}.$ 

\item If there are $2$ $z$'s in $X_{2n-1,2n-1}, X_{2n,2n}, \text{ then }\sigma=\{(a+1)^{j'}, a^{n-j'-k'}, (a-1)^{k'}\}\text{ (summand 5)}$, $\{(a+1)^{j'+1}, a^{n-j'-k'-1}, (a-1)^{k'+1}\}\text{ (summand 6)}$, $ \{(a+1)^{j'+2}, a^{n-j'-k'-2}, (a-1)^{k'}\}\text{ (summand 7)}, \{(a+1)^{j'+1}, a^{n-j'-k'}, (a-1)^{k'-1}\}\text{ (summand 8)}$.
\end{enumerate}
Each summand is the number of $X$ with fixed $\sigma$ times the number of all possible fillings in the last two columns given $\sigma $. To simplify the notation afterward, we will reorder the $XB$'s 2 dimensional blocks in nondecreasing order, and make all $z^a$ blocks with both rows affected by z in front of all the unaffected $z^a$ blocks ($x_1 z^{a-1}\text{'s } , x_2 z^a \text{'s } \text{affected}, x_3 z^a\text{'s } \text{affected}, x_4 z^{a+1}\text{'s } $). We denote 
\begin{enumerate}
    \item $A_1=\{X_{2r,2n-1},\text{where } r\leq x_1\},$
    \item $A_2=\{X_{2r-1,2n-1},\text{where } x_1+x_2<r\leq x_3+x_2+x_1\},$
    \item $A_3=\{X_{2r,2n-1},\text{where } x_1+x_2<r\leq x_3+x_2+x_1\},$
    \item $A_4=\{X_{2r-1,2n-1},\text{where } x_1+x_2+x_3<r\leq x_4+x_3+x_2+x_1\},$
\end{enumerate}

and similarly for $B_1,B_2,B_3,B_4 $ on $X_{r,2n}$.

The following statements are direct results from Proposition~\ref{young}, since these are from the same matrix operations as in case 1,2,3. 

Claim 1: If $X_{2n,2n}=z\text{ and }X_{2n-1,2n-1}=1$, then the last $z^a$ block becomes $z^{a+1}.$ In terms of $B's$, any elements in $k$ can be filled in all entries in $B_1,B_2,B_3,B_4$ (case 1 in Proposition~\ref{young}). 

If $X_{2n,2n}=1\text{ and }X_{2n-1,2n-1}=z$, then:

Claim 2: A $z^a $ block in $[B]$ becomes $z^{a+1}$ (subcase 1 in case 2) and in terms of $B's$, any elements in $k$ can be filled in all entries in $A_1,A_2,A_4$ and there exists at least one nonzero entries in $A_3$ ; 

Claim 3: A $z^{a-1}$ block in $[B]$ becomes $z^{a}$ (subcase 1 in case 2) and in terms of $B's$, any elements in $k$ can be filled in all entries in $A_2,A_4$ and there exists at least one nonzero entries in $A_1$ and all entries in $A_3$ are $0$;

Claim 4: A $z^{a}$ block in $[B]$ becomes $z^{a-1}$ (subcase 2 in case 2) and in terms of $B's$, any elements in $k$ can be filled in all entries in $A_4$ and there exists at least one nonzero entries in $A_2$ and all entries in $A_3,A_1$ are $0$;

Claim 5: The last $z^a$ block  becomes $z^{a-1}$ (subcase 3 in case 2) and in terms of $B's$, any elements in $k$ can be filled in all entries in $A_4$ and all entries in $A_3,A_1,A_2$ are $0$;

If $X_{2n,2n}=z\text{ and }X_{2n-1,2n-1}=z$, then:

Claim 6: A $z^a $ block in $[B]$ becomes $z^{a+1}$ and the last $z^a$ block  becomes $z^{a+1}$ (subcase 1 in case 3) and thus in terms of $B's$, any elements in $k$ can be filled in all entries in $A_1,A_2,A_4,B_1,B_2,B_3,B_4$ and there exists at least one nonzero entries in $A_3$ ; 

Claim 7: A $z^{a-1}$ block in $[B]$ becomes $z^{a}$ and the last $z^a$ block  becomes $z^{a+1}$ (subcase 1 in case 3) and thus in terms of $B's$, any elements in $k$ can be filled in all entries in $A_2,A_4,B_1,B_2,B_3,B_4$ and there exists at least one nonzero entries in $A_1$ and all entries in $A_3$ are $0$;

Claim 8: A $z^{a}$ block in $[B]$ becomes $z^{a-1}$ and the last $z^a$ block  becomes $z^{a+1}$ (subcase 3 in case 3) and in terms of $B's$, any elements in $k$ can be filled in all entries in $A_4,B_1,B_2,B_3,B_4$ and there exists at least one nonzero entries in $A_2-B_3$ and all entries in $A_3,A_1$ are $0$;

Claim 9: The last $z^a$ block  becomes $z^{a}$ (subcase 3 in case 3) and in terms of $B's$, any elements in $k$ can be filled in all entries in $A_4,,B_1,B_2,B_3,B_4$ and all entries in $A_3,A_1,A_2-B_3$ are $0$.

Now, we explain the coefficient for each summand in the equality.

Summand 1: Directly from the $n-1$ blocks with ${i,j,k}$ indices. Since there is no $z$'s on the two columns, there is only one possible filling.

Summand 2: From the $n-1$ blocks with ${i,j,k-1}$ indices, the last block provides one $-1$ (not necessarily on that block.) Equivalently, claim 4 and 5 gives $A_1,A_3=0$, and fill $A_2, A_4$ with $k$.

Summand 3: From the $n-1$ blocks with ${i,j-1,k}$ indices, the last block provides one $+1$ (not necessarily on that block). 
We distinguish two cases below. 
\begin{enumerate}
    \item  $X_{2n-1,2n-1}=z$ : fill all$A_3$ with at least one nonzero element and fill $A_2, A_4, A_1$ with k. This summand becomes $q^{n-i-1}(q^{n-i-j-k}-1)$.
    \item $X_{2n,2n}=z$, fill $B_1,B_2,B_3,B_4$ with $k$ (Claim 1). This summand becomes $q^{2n-k-j-2i}.$
\end{enumerate}

Summand 4: From the $n-1$ blocks with ${i-1,j,k+1}$ indices, the last block takes one $-1$ block to $\pm 1$. Equivalently, by claim 3 we can fill all entries fill all $A_1$ with at least one nonzero element, $A_3=0$, fill $A_2,A_4$ with k. This summand becomes $q^{n-i-k-1}(q^{k+1}-1)$.

Summand 5: From the $n-1$ blocks with ${i-1,j,k}$ indices, the last block provides $\pm 1$. Equivalently, by claim 9 we can fill all entries $A_4,B_1,B_2,B_3,B_4$ with $k$ and $A_3,A_1,A_2-B_3$ with $0$. This summand becomes $q^{2n-2i-k}$.

Summand 6: From the $n-1$ blocks with ${i,j-1,k-1}$ indices, the last block provides  $+1\text{and } -1$. Equivalently, by claim 8 we can fill all entries $A_4,B_1,B_2,B_3,B_4$ with $k$ and fill all$A_2$ with at least one nonzero element in $A_2-B_3$. This summand becomes $q^{2n-2i-k-1}(q^{n-i-j-k+1}-1)$.

Summand 7: From the $n-1$ blocks with ${i,j-2,k}$ indices, the last block provides  $+1,+1$. Equivalently, by claim 6 we can fill all entries $A_3$ with at least one nonzero element and fill $A1,A_2,A_4,B_1,B_2,B_3,B_4$ with $k$.  This summand becomes  $q^{3n-3i-j-k-1}(q^{n-i-j-k+1}-1)$.

Summand 8: From the $n-1$ blocks with ${i-1,j-1,k+1}$ indices, the last block provides  $+1,\pm1$. Equivalently, by claim 7 we can fill all entries $A_1$ with at least one nonzero element and fill $A_2,A_4,B_1,B_2,B_3,B_4$ with $k$. This summand becomes $q^{3n-3i-2k-j-1}(q^{k+1}-1)$.

Next, we group some of these summands to reduce the redundancy.

(2,6): $q^{(n-1)(i+j)-i(i+j+k)} \frac{[n-1]_{i+j+k-1}}{[i]![j]![k]!}q^{j+n-k}[k]$

(3,7): $q^{(n-1)(i+j)-i(i+j+k)} \frac{[n-1]_{i+j+k-1}}{[i]![j]![k]!}(q^{n-i-k}+q^{n-i-j-k}-1)[j]$

(4,5,8): $q^{(n-1)(i+j)-i(i+j+k)} \frac{[n-1]_{i+j+k-1}}{[i]![j]![k]!}(q^{n-i-k}+q^{n-1}+1)[i]$

1: $q^{(n-1)(i+j)-i(i+j+k)} \frac{[n-1]_{i+j+k-1}}{[i]![j]![k]!}[n-i-j-k]$.

From direct computation, we have \[q^{j+n-k}[k]+(q^{n-i-k}+q^{n-i-j-k}-1)[j]+(q^{n-i-k}+q^{n-1}+1)[i]+[n-i-j-k]=q^{i+j}[n].\]
Thus the statement is true for $n.$
\end{proof}
In the following proposition, we have $\nu=\{1^{n_2},0^{n_1}\}$. Equivalently, on $B$, it consists of 1-dim blocks.
Recall that we define $h_{-1^{i+j+k},\{1^{n_2},0^{n_1}\}}^{2^k,1^{n_2+i-j-k}} \text{ as }\bar{h}_1^{\omega_{[1]}}$ where $\omega_{[1]}=(k,j,i,0).$  

\begin{proposition}

\label{compute1}

\[h_{-1^{i+j+k},\{1^{n_2},0^{n_1}\}}^{2^k,1^{n_2+i-j-k}}=q^{\frac{(k-1)k}{2}}(q-1)^k[n_1-i]_{k}{\binom{n_2}{k,j,n_2-k-j}}\binom{n_1} {i}\]
\end{proposition}
\begin{proof}

Let $B$ denote the diagonal matrix  where $B_{l,l}=1 \text{for } l\leq n_1 \text{ and }  B_{l,l}=z \text{ for } n_1<l\leq n_1+n_2$. Step 1: for all $l\text{ such that } X_{l,l}=1,$ do column operations to remove all $X_{l,r}B_{r,r}$. Thus, $Y_{l,r}=0 \text{ if } l\neq r,\text{and }( l,r\leq n_1 \text{or } n_1<l,r\leq n_1+n_2)$. 
Step 2, for all $l \text{ such that } X_{l,l}=z, \text{and } Y_{l,r}\neq 0$, pick the first nonzero $Y_{l,r}$ in the $r$th column, then remove every $Y_{l',r}$ below (i.e., $l'<r$). Then, remove all $Y_{l,p}$ for $p>r$.  We call $(l,r)$ a pair.
Step 3, for all $l \text{ such that } X_{l,l}=z, \text{ and } Y_{l,r}=0$, this row becomes $1$.

On column $l<n_1$, if there is a $z$ action on the column, then $1$ becomes $z $ with weight $q^{l-1-r'(l)}$. On column $n_1<l<n_1+n_2$, with a $z$ action, either a $z^2$ or 1 will occur. We write $J=\{j_1, j_2, ...j_j\}$ to be those columns  with $z\to 1$, $K=\{k_1, ...k_k\}$ to be those columns with $z\to z^2$, $\{i_1,...i_k\}$  to be those columns paired with $k$'s, $I=\{i'_1, i'_2, ...i'_i\}$ to be those columns with $1\to z$,.
A $z\to z^2$ occurs when there exists a $l'<n_1$ column of 1. Given $J,K$, for any l$\in J,K$, we have  $r(l)=|\{k_p<l\}\cup \{j_p<l\}|$. Given $I$, for any l$\in I$, we have $r'(l)=|\{i_p<l\}|$. 

The total weight for $k_l$ is $(q^{n_1-i}-q^{l-1})q^{k_l-n_1-1-r(k_l)}$. (The $q^{l-1}$ term is the weight for $ i's, \text{ where }z\to z^2$). The weight for $j_l$ is $q^{j_l-n_1-1-r(j_l)}q^{|\{ k_{l'}\text{ ,where }k_{l'}<j_l\}|}$.
\begin{align*}
h_{-1^{i+j+k},\{1^{n_2},0^{n_1}\}}^{2^k,1^{n_2+i-j-k}}=&\sum_{I,J,K} \prod_{l} q^{i'_l-1-r'(i'_l)}\prod _{l} (q^{n_1-i}-q^{l-1})q^{k_l-n_1-1-r(k_l)} \prod_{l} q^{j_l-n_1-1-r(j_l)}q^{|\{ k_{l'}\text{ ,where }k_{l'}<j_l\}|}  \\
&=\binom{n_1} {i}\sum_{J,K}\prod _{l} (q^{n_1-i}-q^{l-1})q^{k_l-n_1-1-r(k_l)} \prod_{l} q^{j_l-n_1-1-r(j_l)}q^{|\{ k_{l'}\text{ ,where }k_{l'}<j_l\}|}\\
&= \binom{n_1} {i} \binom{n_2}{k+j} \prod_{l=0}^{k-1} (q^{n_1-i}-q^l) \cdot \sum_{J,K} \prod_l q^{|\{ k_{l'}\text{ ,where }k_{l'}<j_l\}|}\\
\end{align*}
The  last summation is exactly the combinatorial description of $q$-binomial coefficient in terms of inversion. Thus 
\[h_{-1^{i+j+k},\{1^{n_2},0^{n_1}\}}^{2^k,1^{n_2+i-j-k}}= \binom{n_1}{ i} \binom{n_2}{ k+j}  \binom{k+j}{ k}\prod_{l=0}^{k-1} (q^{n_1-i}-q^l).\]By rearranging, the original formula holds.
\end{proof}

The next step is to construct a system of paired semi-tableaux to prove the original theorem, i.e., closing the gap of $\nu_{i+1}-\nu_i >4$. The construction is pure combinatorial.
We will define $ D_{l,\nu}^\lambda $ (the set of all paired diagrams which the original diagrams are $\nu$ and reduced shapes are $\lambda$  with the total amount of added boxes being $l$) and 
$T_\omega$ (the set of all paired semi-tableau from a paired diagram $\omega$) before stating Proposition ~\ref{young}.

For any Young diagram $\nu$ with $m_0(\nu)=n/2\text{ and } m_1(\nu)=0$, we define an approach to add box toward $\nu$. New boxes are added to $\nu$ following the original Pieri rule, i.e. at most $1$ box per row. We call one such added box diagram $\omega$ (We do not reshuffle rows in $\omega$ to make them in decreasing order). We introduce a pairing $(i,j)$ for rows $i,j$: for any i with $\nu_i > 1$, there is a pairing for $\nu_i$ with $\nu_{n-i+1}$, i.e., $(i,n-i+1)$. The total number of added boxes is $\ell$. 
Let $r(\omega)=\{r_1,r_2,...\}$ be the reduced shape of $\omega$: 
$r_i= \omega_i-\omega_j \text{ for pair } (i,j)$. Note that $r(\omega)$ is a partition.
We call all such Young diagrams with added boxes {\em paired diagrams},  with the original diagram $\nu$ and reduced shape $\lambda=r(\omega)$  and  total amount of added boxes being $\ell$. Let $D_{l,\nu}^\lambda $ denote the set containing all such paired diagrams.

Next, a definition of {\em paired semi-tableaux} based on paired diagram $\omega$ is introduced.
For each ADDED box in $\omega$, numbers in $\{1, ..., n\}$ will be filled in. Notice that in the original definition of Young tableaux, all boxes are filled with numbers.
Rules:
\begin{enumerate}
\item Each number can appear in at most one box.
\item The number $2k+1$ can be only filled in the $(k+1)$th row.
\item The number $2k$ can be filled in any row from the $(k+1)$th row to the $n-k+1$th row with the exception: $2k$ cannot be filled in row $j$ if $\nu_{k}=\nu_{n-j+1}\text{ and } \nu_{j}=0$ and the number filled in row $n-j+1$ is greater than $2k$. 
\end{enumerate}

Before we start Proposition~\ref{young}, a definition of weight on the paired semi-tableau $\gamma$ would be necessary.
We first assign a label $(r(\omega)_i, n-i)$  for each row of $\gamma$, and a lexicographic order on labels.
We create a sequence of subtableaux of $\{[\gamma]_1,...,[\gamma]_{\frac{n}{2}}\}$, where $[\gamma]_i$ consists of $[\frac{n}{2}-i+1,\frac{n}{2}+i]$ rows of $\nu$ with all added boxes with filled number $\geq n-2i+1$. Naturally, labels are assigned to rows in $[\gamma]_i$.
Let $S_{(a,b),[\gamma]_i}$ denote the set of all rows in $[\gamma]_i$ with no added boxes in $[\gamma]_i$ and labels being less than $(a,b)$. Here is the definition for weight on $\{1,...,n\}$:
\begin{enumerate}
    \item For $2k+1$ newly filled at $[\gamma]_i$ ($2k+1$ does not appears in $[\gamma]_{i-1}$, that is $i=\frac{n}{2}-k$) , $ wt_{2k+1}(\gamma)=q^{|S_{(\infty,n),[\nu]_i}|}$.
    
\item For $2k$ newly filled at $[\gamma]_i$ at the $j$th row (that is  $i=\frac{n}{2}-k+1$), we write the label for row j as $(a,n-j)$. Then $wt_{2k}(\gamma)=q^{|S_{(a-1,n-j),[\gamma]_i}|-1}(q-1)$ if $j\neq n-k-1$.
\item Under the same condition with 2, if $j= n-k-1$ and $2k-1$ is not in $[\gamma]_i$ , $wt_{2k}(\gamma)=q^{|S_{(a-1,n-j),[\gamma]_i}|}$.
\item  Under the same condition with 2, if $j= n-k-1$ and $2k-1$ is in $[\gamma]_i$, $wt_{2k}(\gamma)=q^{|S_{(a,n-j),[\gamma]_i}|}$.
\item If a number $k$ is not filled in $\gamma$, $wt_{k}(\gamma)=1$. 
\end{enumerate}
Then, we define a weight on tableau $\gamma$: $wt(\gamma)=\prod_{i} wt_{i}(\gamma)$.

Here is an example.

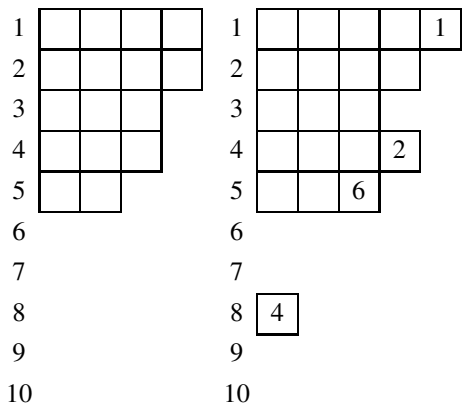
\ytableausetup{centertableaux}
\begin{ytableau}
\none[1]  &  &  &  & \\
\none[2]  &  &  &  &\\
\none[3]  &  &  &  \\
\none[4]  &  &  &  \\
\none[5] & & \\
\none[6] \\
\none[7] \\
\none[8] \\
\none[9] \\
\none[10] \\
\end{ytableau}
\ytableausetup{centertableaux}
\begin{ytableau}
\none[1]  &  &  &  &  &1\\
\none[2]  &  &  &  &\\
\none[3]  &  &  &  \\
\none[4]  &  &  &  &2\\
\none[5] & & &6\\
\none[6] \\
\none[7] \\
\none[8] &4\\
\none[9] \\
\none[10] \\
\end{ytableau}

With pair $\{(1,10),(2,9),(3,8),(4,7)(5,6)\}$, we write  subtableaux sequence of $\gamma$  $\{[\gamma]_1,[\gamma]_2,[\gamma]_3,[\gamma]_4,[\gamma]_5\}.$

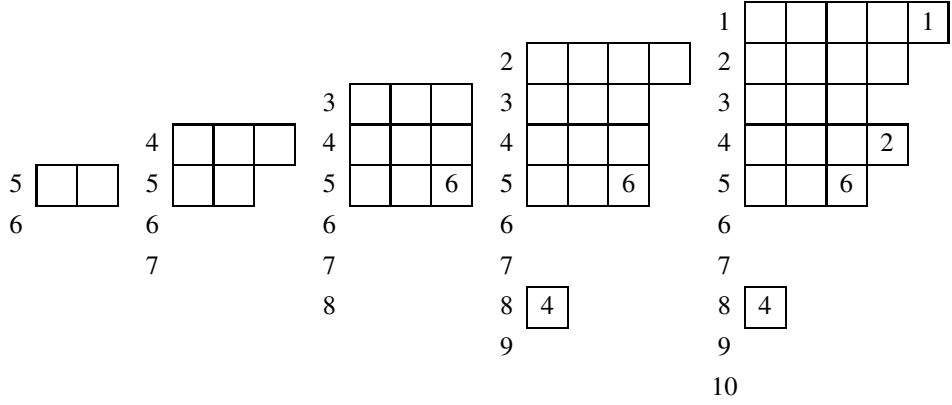
\ytableausetup{centertableaux}
\begin{ytableau}
\none[5] & & \\
\none[6] \\
\end{ytableau}
\ytableausetup{centertableaux}
\begin{ytableau}
\none[4]  &  &  & \\
\none[5] & & \\
\none[6] \\
\none[7] \\
\end{ytableau}
\ytableausetup{centertableaux}
\begin{ytableau}
\none[3]  &  &  &  \\
\none[4]  &  &  & \\
\none[5] & & &6\\
\none[6] \\
\none[7] \\
\none[8] \\
\end{ytableau}
\ytableausetup{centertableaux}
\begin{ytableau}
\none[2] &  &  &  &\\
\none[3]  &  &  &  \\
\none[4]  &  &  & \\
\none[5] & & &6\\
\none[6] \\
\none[7] \\
\none[8] &4\\
\none[9] \\
\end{ytableau}
\ytableausetup{centertableaux}
\begin{ytableau}
\none[1]  &  &  &  & &1\\
\none[2]  &  &  &  &\\
\none[3]  &  &  &  \\
\none[4]  &  &  & &2\\
\none[5] & & &6\\
\none[6] \\
\none[7] \\
\none[8] &4\\
\none[9] \\
\none[10] \\
\end{ytableau}

The reduced shape is $\{5,4,4,3,2\}.$

$wt_6=(q-1)q^2$. Labels in $[\gamma]_3: ((3,7),(3,6),(3,5),(-3,4),(-3,3),(-3,2))$, 

$wt_4=(q-1)$. Labels in $ [\gamma]_4: ((4,8),(2,7),(3,6),(3,5),(-3,4),(-3,3),(-2,2),(-4,1)),$

$wt_{2}=(q-1)q^4,wt_{1}=q^6.$ Labels in $[\gamma]_5$: $((5,9),(4,8),(2,7),(4,6),(3,5),(-3,4),(-4,3), $ $(-2,2),(-4,1),(-5,0)).$

The following proposition will give a combinatorial summation for $h_{\ell\nu}^{\lambda}$. Since in Lemma~\ref{decomp_01} we completely separated 1-dimensional blocks and 2-dimensional blocks, $\nu$ has the property $m_0(\nu)=\frac{n}{2}\text{ and } m_1(\nu)=0$. Again, unlike Lemma~\ref{decomp_01},~\ref{decomp_flip} and ~\ref{weak}, we denote B to be the block diagonal matrix with valuations in nondecreasing order and $\sigma(B)=\nu$. Recall from the beginning of this chapter, this is the definition of $\Pi^{\nu}$.
\begin{proposition}
    
\label{young}
  \[h_{-1^\ell\nu}^{\lambda}=\sum_{\omega \in D_{\ell,\nu}^\lambda} \sum_{ \gamma\in T_\omega}wt(\gamma).\] 
\end{proposition}

\begin{proof}
Recall the key property for $X\in \mathcal{A}_\ell$:
Given any i such that $X_{i,i}=z,$ we have $X_{i,j}=0$ for $j>i$ and $X_{j,i}\in k$ for $j<i.$ Given any i with $X_{i,i}=1$, we have $X_{i,j}\in{k} $  for $j>i$ and $X_{j,i}=0$ for $j<i.$ 

We claim that each $X$ can be represented uniquely by a paired semi-tableau. Moreover weight of each paired semi-tableau is the number of all such $X$'s corresponding to the tableau. The algorithm below shows that for any $i$ such that $X_{i,i}=z$, first i by i submatrix of $XB$ (if $i$ is the second row of some two dimensional blocks of $B$)  or first  $i+1$ by $i+1$ submatrix of $XB$ ( if i is the first row of some two dimensional blocks of $B$) has one block turning from $z^r$ to $z^{r-1}\text{ or } z^{r+1}$ comparing to the $(i-2, i-2)$ submatrix (former) or $(i-1,i-1)$ submatrix (latter). This is equivalent to filling $n-i+1$ in to some added box under the restriction of filling boxes. 

For any $i$ such that $X_{i,i}=z$, we take $j$ to be the number of elements in the $i$th column vector that are not necessarily 0. In the algorithm we will present, we call $\tilde{X}_i$ to be ith column vector of the matrix corresponding to $Y$ in the $X$ form. 
Thus there are $q^j$ choices for this column vector. In the algorithm we present below, the current column vector correspond to a matrix in $\text{GL}_{i-1}(k[[z]])$ times the original vector in $X$. Thus, the map $k^j\to k^j$ is bijective. We do not necessarily need the following statement (since the map is bijective), but the algorithm will follow the statement:  All elements that must be 0s (originated from the rule in $X$) still remain to be $0$ under some row operations.

We introduce 3 tricks that are used below. We call left multiplication by matrices in $\text{GL}_n(k[[z]])$ and right multiplication by matrices in $\text{GL}_n(k((z^2)))$ allowed operations.

Trick 1: 
A submatrix of the form $\begin{pmatrix}
1 & z+z^{n_1}b\\
0 & z^{n_1+1}
\end{pmatrix}$ (the column with the 1 only has one nonzero entry, $b\in k$ ,$n_1>1$) can be reduced to $\begin{pmatrix}
1 & z\\
0 & z^{n_1+1}
\end{pmatrix}$  with allowed operations.
If $2|n_1$, right multiply by $\mathbbm{1}_{(1,2)}(-bz^n_1)$; else, left multiply by $\mathbbm{1}_{(1,1)}(1/(1+bz^{n_1-1}))$ and right multiply by $\mathbbm{1}_{(1,1)}(1+bz^{n_1-1})$. Note that this may change the other entries in the same row as the top row of the submatrix.

Trick 2: 
A submatrix of the form $\begin{pmatrix}
1 & z & 0&f\\
0 & z^{n_1} &0 &g\\
0 & 0 & 1&z \\
0 & 0 &0 &z^{n_2}\\
\end{pmatrix}$ (the first two columns only do not have other nonzero entries, $v(f)\geq 0$) can be reduced to $\begin{pmatrix}
1 & z & 0&0\\
0 & z^{n_1} &0 &g\\
0 & 0 & 1&z \\
0 & 0 &0 &z^{n_2}\\
\end{pmatrix}$  with allowed operations.

We left multiply by $\mathbbm{1}_{(1,3)}(-f_{\text{odd}}/z)$, right multiply by $\mathbbm{1}_{(1,3)}(-f_{\text{odd}}/z)$ $\mathbbm{1}_{(1,4)}(-f_{\text{even}})$.
Note that this may change the other entries in the same row as the top row of the submatrix. This trick still holds if the third row is in the form of $z,z^2$, given $v(f)>1.$

Trick 3: 
A submatrix of the form  $\begin{pmatrix}
1&z&0&g\\
0&z^{a_1}&0&f\\
0&0&1&z\\
0&0&0&z^{n_1}\\
\end{pmatrix}$ (the first two columns do not have other nonzero entries, $v(f)\geq a_1,v(g)\geq 0$) can be reduced to $\begin{pmatrix}
1&z&0&0\\
0&z^{a_1}&0&0\\
0&0&1&z\\
0&0&0&z^{n_1}\\
\end{pmatrix}$,  with allowed operations.
If $2|a_1$, left multiply by $\mathbbm{1}_{(1,3)}(-f_{odd} z^{-a_1})$$\mathbbm{1}_{(2,3)}(-f_{odd}/z)$ right multiply by$\mathbbm{1}_{(2,4)}(-f_{even} z^{-a_1})\mathbbm{1}_{(2,3)}(-f_{odd} z^{-a_1-1}) $, then use trick 2. For $a_1$ being odd, vice versa. Note that this may change the other entries in the same row as the top two rows of the submatrix.

Next we present an algorithm on decomposing $Y=XB$. 
We apply the algorithm in all $i$'s such that $X_{i,i}=z$, in increasing order. Then there are three cases:
$v(B_{i,i})>1$;
$v(B_{i,i})=0 $ and $v(B_{i+1,i+1})>1 $  and $X_{i+1,i+1}=1$;
$v(B_{i,i})=0$ and $v(B_{i+1,i+1})>1 $ and $X_{i+1,i+1}=z$.

Case 1: $v(B_{i,i})>1$.

$\begin{pmatrix}
1&z&0&\cdots&b_1 z^{a_l}\\
0&z^{a_0}&0&\cdots&b_2 z^{a_l}\\
\cdots\\
0&\cdots&\cdots&1&z+b_r z^{a_l}\\
0&\cdots&\cdots&0&z^{a_l+1}\\
\end{pmatrix}$

By trick 1, we remove $b_r z^{a_l}$. 
By trick 2, we remove $b_1 z^{a_l}$. 
By trick 3, we remove $b_2 z^{a_l}$.

At the block corresponding to $i$, $z^{a_l}\to z^{a_l+1} $
Thus, weight of the tableau of this $X$ at $i$ is $q^{i-y}$, where $y$ is the number of rows above $i$ that are permanently 0. On the tableau, this is equivalent to fill $n-i+1$ in the added box of row $\frac{n-i}{2}+1$.  In the notation of lexicographical order, $|S_{(\infty,n),[\gamma]_{i/2}}|$ is the number of all rows $<i$ in $Y$ with possible nonzero entries. Thus, $wt_{n-i+1}=q^{|S_{(\infty,n),[\gamma]_{i/2}}|}$.

Case 2: $v(B_{i,i})=0 $ and $v(B_{i+1,i+1})>1 $  and $X_{i+1,i+1}=1$;

$\zeta_l=
    \begin{cases}
      (-v(Y_{l+1,l+1}),\frac{n-l-1}{2}) , & \text{if } v(Y_{l,l})=0 \text{ and }l\leq 1+i\\
      (v(Y_{l,l}),\frac{l+n}{2}-1), &\text{if } v(Y_{l,l})>0 \text{ and }l\leq i+1\\

    \end{cases}$

We write $\zeta(Y)=(\zeta_1,\zeta_2,...\zeta_{i+1})$ with  lexigraphical order imposed on $\zeta$'s.

There are three subcases based on $Y_{l,i}$: 
1. There exists $ l \text{ such that } Y_{l,i}\neq 0\text{ with } v(B_{l,l})>1$; 
2. For any $l\text{ such that } v(B_{l,l})>1\text{ we have } Y_{l,i}= 0; \text{ There exists }  l \text{ such that } Y_{l,i}\neq 0 \text{ where }B_{l,l}=1,v(B_{l+1,l+1})\leq v(B_{i+1,1+i}).$
3. For any $l\text{ such that } v(B_{l,l})>1 \text{ or } B_{l,l}=1,v(B_{l+1,l+1})\leq B_{i+1,i+1},\text{ we have } Y_{l,i}= 0; $  

Subcase 1: There exists $ l \text{ such that } Y_{l,i}\neq 0\text{ with } v(B_{l,l})>1$.
Among all nonzero $ Y_{2l,i}$,  we take $2l'$ to be the index of the largest $\zeta_{2l}$ with respect to  $ \zeta(Y) $.

$\begin{pmatrix}
1&z&0&0&b_1& b_1 z\\
0&z^{a_1}&0&0&b_2& b_2 z\\
0&0&1&z&b_{2l'-1}& b_{2l'-1} z\\
0&0&0&z^{a_{l'}}&b_{2l'}& b_{2l'} z\\
0&0&0&0&z&z^2 \\
0&0&0&0&0&z^{a_{i+1/2}}\\
\end{pmatrix}$

Remove terms with $b_{2l'-1}$: left multiply by $\mathbbm{1}_{(3,4)}(-b_{2l'-1}/b_{2l'})$, then trick 1.
Remove all $b_{2m-1}$'s (i.e., in this matrix, remove $b_1$): left multiply by $\mathbbm{1}_{(1,4)}(-b_{1}/b_{2l'})$, then trick 2.
Remove all $b_{2l}$'s (i.e., in this matrix, remove $b_2$): left multiply by $\mathbbm{1}_{(2,4)}(-b_{2}/b_{2l'})$, then trick 3 (since we get rid of $b_{2l'-1}$ first, row 1 and 2 are not affected in the columns $5,6$). 
Left multiply by $\mathbbm{1}_{(5,4)}(-z/b_{2l'})$, then trick 2 will take $Y_{4,4}$ to be 0.
Reorder the row of  $z,z^2$ with the row of $b_{2l'}$:  switch row $4,5$ and left multiply by $\mathbbm{1}_{(4,4)}(1/b_{2l'})$. Notice that the current row of $z^{a_{l'}+1}$ is obtained by $z$ times that of the original row. Thus, any terms on the right will be removed to $0$.

Therefore, $z^{a_{l'}}\to z^{a_{l'}+1}$ and the other entries in the corresponding row become permanent 0. On the tableau, this is equivalent to fill $n-i+1$ in the added box of row $n/2-l'+1$.
In the notation of lexicographical order, $S_{(a_{l'},l'+\frac{n}{2}-1),[\gamma]_{i+1/2}}$ consists of all rows in $Y$ above $i$ with possible nonzero entries with $\zeta$ less than $(a_{l'},l'+\frac{n}{2}-1)$ and row $i$.
Thus, $wt_{n-i+1}(\gamma)=q^{|S_{((a_{l'},l'+\frac{n}{2}-1),[\gamma]_{(i+1)/2}}|-1}(q-1)$.

Subcase 2: For any $l\text{ such that } v(B_{l,l})>1,\text{ we have } Y_{l,i}= 0. \text{ There exists }  l \text{ such that } Y_{l,i}\neq 0 \text{ where }B_{l,l}=1,v(Y_{l+1,l+1})\leq v(B_{1+i,1+i}).$
Among all nonzero $ Y_{2l-1,i}$,  we take $2l'-1$ to be the index of the largest $\zeta_{2l-1}$ with respect to  $ \zeta(Y) $.

$\begin{pmatrix}
1&z&0&0&b_{2l'-1}& b_{2l'-1} z\\
0&z^{a_{l'}}&0&0&0&0\\
0&0&1&z&b_{2l-1}& b_{2l-1} z\\
0&0&0&z^{a_{l}}&0 &0\\
0&0&0&0&z&z^2 \\
0&0&0&0&0&z^{a_{i+1/2}}\

\end{pmatrix} $

Remove $b_{2l}$: left multiply by $\mathbbm{1}_{(4,2)}(-z^{a_{l}-a_{l'}}b_{2l-1}/b_{2l'-1})\mathbbm{1}_{(3,1)}(-b_{2l-1}/b_{2l'-1})$, right multiply by$\mathbbm{1}_{(3,1)}(b_{2l-1}/b_{2l'-1})\mathbbm{1}_{(4,2)}(b_{2l-1}/b_{2l'-1})$.
Remove the row $z,z^2$ (we delete row 3,4): left multiply by $\mathbbm{1}_{(2,2)}(b_{2l'-1})\mathbbm{1}_{(2,3)}(z^{a_{l'}-2}b_{2l'-1})\mathbbm{1}_{(4,2)}(z^{a_{i+1/2}-a_{l'}}/b_{2l'-1})\mathbbm{1}_{(3,1)}(z/b_{2l'-1})$, right multiply by $\mathbbm{1}_{(1,3)}(-b_{2l'-1})\mathbbm{1}_{(2,4)}(-b_{2l'-1})\mathbbm{1}_{(4,2)}(-1/b_{2l'-1}) \mathbbm{1}_{(3,1)}(-1/b_{2l'-1}) \mathbbm{1}_{(4,4)}(-1/z^2)(3,4) $. Now, we permute rows and columns correspondingly to make $z^{a_{l'}-1}$ block still in row $2l'-1,2l'$.

Therefore, it takes $z^{a_{l'}}\to z^{a_{l'}-1} $and the corresponding row becomes a permanent 0 row. 
On the tableau, this is equivalent to fill $n-i+1$ in the added box of row $\frac{n}{2}+l'$.
In the notation of lexicographical order, $S_{(-a_{l'},-l'+\frac{n}{2}),[\gamma]_{(i+1)/2}}$ consists of  all rows above $i$ with possible nonzero entries with $\zeta$ less than $(-a_{l'},-l'+\frac{n}{2})$ and row $i $ ($v(B_{(i+1)/2,(i+1)/2})\geq a_{l'}$). Thus, 
$wt_{n-i+1}(\gamma)=q^{|S_{(-a_{l'},-l'+\frac{n}{2}),[\gamma]_{i+1/2}}|-1}(q-1)$.

Subcase 3: For any $l\text{ such that } v(B_{l,l})>1 \text{ or } B_{l,l}=1,v(Y_{l+1,l+1})\leq v(B_{i+1,i+1}),\text{ we have } Y_{l,i}= 0.$  
Notice that this may be different from $(Y_{1,i},Y_{2,i},...,Y_{i-1,i})=(0,0,..,0)$. On the $2l-1$th row of $Y_{2l-1,i}\neq 0 \text{ and } Y_{2l,2l}=zY_{i,i}$:  (since the original $B$ is in increasing order, row $2l$ of $ Y$ is a permanent 0 row).

$\begin{pmatrix}
1&z&b_{2l}& b_{2l} z\\
0&z^{a_{l}+1}&0 &0\\
0&0&z&z^2\\
0&0&0&z^{a_{i+1/2}}\\

\end{pmatrix} $

Remove the row $z,z^2$: left multiply by $\mathbbm{1}_{4,4}(-1)\mathbbm{1}_{4,3}(-z^{a_{2l'}-2})\mathbbm{1}_{2,4}(zb_{2l})$, right multiply by $\mathbbm{1}_{1,3}(-b_{2l})\mathbbm{1}_{2,4}(-b_{2l})\mathbbm{1}_{4,4}(z^{-2})$.  Then switch columns $3,4$.

This takes $z^{a_l}$ to $z^{a_l-1}$ and the corresponding row becomes a permanent 0 row. On the tableau, this is equivalent to fill $n-i+1$ in the added box of row $\frac{n+i-1}{2}$.
In the notation of lexicographical order, $S_{(-a_{i+1/2},\frac{n-i+1}{2}),[\gamma]_{i+1/2}}$ consists of all rows $\leq 1+i$ in $Y$ with possible nonzero entries with $\zeta$ less than $(-a_{i+1/2},\frac{n-i+1}{2})$ (it is equivalent to first rows with $z^{1+a_{i+1/2}}$). Thus, 
$wt_{n-i+1}(\gamma)=q^{|S_{(-a_{i+1/2},\frac{n-i+1}{2}),[\gamma]_{i+1/2}}|}$.
Moreover, it shines half light on the rule why no added box on the first row of $z^{a_l+1}$.

Case 3:
$v(B_{i,i})=0$ and $v(B_{i+1,i+1})>1 $ and $X_{i+1,i+1}=z$.

This will be similar to case 2 combined with case 1, but with some subtle differences in subcase 3. Thus, the definition of $\zeta$ applies here. $X_{l,i+1}$ does not play a role until subcase 3. 

Subcase 1: There exists $ l \text{ such that } Y_{l,i}\neq 0\text{ with } v(B_{l,l})>1$.
It is equivalent to first do subcase 1 of case 2 then do case 1. 
Among all nonzero $ Y_{2l,i}$,  we take $2l'$ to be the index of the largest $\zeta_{2l}$ with respect to  $ \zeta(Y) $.

Therefore, $z^{a_{l'}}\to z^{a_{l'}+1} $,$z^{a_i}\to z^{a_i+1} $ and the corresponding rows become permanent 0 row.
On the tableau, this is equivalent to fill $n-i+1$ in the added box of row $n/2-l'+1$ and n-i in the added box of row $\frac{n-i+1}{2}$.
Similarly to subcase 1 in case 2 and case 1, 
$wt_{n-i+1}(\gamma)=q^{|S_{((a_{l'},l'+\frac{n}{2}-1),[\gamma]_{(i+1)/2}}|-1}(q-1)$; $wt_{n-i}=q^{|S_{(\infty,n),[\gamma]_{i+1/2}}|}$, 
  
Subcase 2: For any $l\text{ such that } v(B_{l,l})>1,\text{ we have } Y_{l,i}= 0. \text{ There exists }  l \text{ such that } Y_{l,i}\neq 0 \text{ where } (B_{l,l}=1,v(Y_{l+1,l+1})< v(B_{1+i,1+i}) ) \text{or} (v(Y_{l+1,l+1})=v(B_{1+i,1+i}) \text{ and } v(B_{l+1,l+1})=v(B_{i+1,i+1})-1).$
Among all nonzero $ Y_{2l-1,i}$,  we take $2l'-1$ to be the index of the largest $\zeta_{2l-1}$ with respect to  $ \zeta(Y) $. (Recall in case 2's subcase 3 had the condition $v(Y_{l+1,l+1})\leq v(B_{1+i,1+i})$).

This will take $z^{a_{l'}}\to z^{a_{l'}-1} , z^{a_i}\to z^{a_i+1},$ and the corresponding rows become permanent 0 row. On the tableau, this is equivalent to fill $n-i+1$ in the added box of row $\frac{n}{2}+l'$ and $n-i$ in the added box of row $\frac{n-i+1}{2}$.
Similarly to subcase 2 in case 2 and case 1, 
$wt_{n-i+1}(\gamma)=q^{|S_{(-a_{l'},-l'+\frac{n}{2}),[\gamma]_{i+1/2}}|-1}(q-1)$; $wt_{n-i}=q^{|S_{(\infty,n),[\gamma]_{i+1/2}}|}$.

Subcase 3: For any $l\text{ such that } v(B_{l,l})>1 \text{ or } B_{l,l}=1,v(Y_{l+1,l+1})\leq v(B_{i+1,i+1}),\text{ we have } Y_{l,i}= 0$ (except the case that $v(Y_{l+1,l+1})=v(B_{l+1,l+1})=v(B_{i+1,i+1})$).

$\begin{pmatrix}
1&z&0&0&0&b_1z^a\\
0&z^{c}&0&0&0&b_2z^a\\
0&0&1&z &b &bz+b_3z^a\\
0&0&0&z^a&0&b_4z^a\\
0&0&0&0&z&z^2 \\
0&0&0&0&0&z^{a+1}\\
\end{pmatrix} 
\rightarrow
\begin{pmatrix}
1&z&0&0&0&0\\
0&z^{a}&0&0&0&b_2z^a\\
0&0&1&z &0 &0\\
0&0&0&z^a&0&(b_4-b)z^a\\
0&0&0&0&z&z^2 \\
0&0&0&0&0&z^{a+1}\\
\end{pmatrix} $ 

By trick 2, $b_1z^a,b_3z^a$ go to 0. This takes $z^a\to z^{a}$ on col 5,6 if $b_2=b_4-b=0$. Also  row $5,6$ are permanent 0 rows.

Take $2l'-1$ to be the index of largest entries in $ \zeta(Y) $ with  $\tilde{X}_{2l',i+1}-Y_{2l'-1,i}\neq 0 \text{ and } Y_{2l',2l'}=z^a $.Recall that $\tilde{X}$ is the corresponding X with current $Y$.
This will take $z^{a_{l'}}\to z^{a_{l'}-1} , z^{a_i}\to z^{a_i+1}$ and the corresponding rows become permanent 0 row.  On the tableau, this is equivalent to fill $n-i+1$ in the added box of row $\frac{n}{2}+l'$ and $n-i $  in the added box of row $\frac{n-i+1}{2}$.
In the notation of lexicographical order, $S_{(-a_{l'},-l'+\frac{n}{2}),[\gamma]_{(i+1)/2}}$ consists of  all rows $<i$ with possible nonzero entries with $\zeta$ less than $(-a_{l'},-l'+\frac{n}{2})$ and row i ($v(B_{(i+1)/2,(i+1)/2})= a_{l'}$). 
In the notation of lexicographical order, $|S_{(\infty,n),[\gamma]_{i+1/2}}|$ is the number of all rows above i+1 with possible nonzero entries. 
Thus, 
$wt_{n-i+1}(\gamma)=q^{|S_{(-a_{l'},-l'+\frac{n}{2}),[\gamma]_{i+1/2}}|-1}(q-1)$; $wt_{n-i}=q^{|S_{(\infty,n),[\gamma]_{i+1/2}}|}$.

If no $l'$ exists, $ z^{a_i}\to z^{a_i}$ and the corresponding two rows both become permanent 0 row. On the tableau, this is equivalent to fill $n-i+1$ in the added box of row $\frac{n+i-1}{2}$ and $n-i$ in the added box of row $\frac{1+n-i}{2}$.
Similarly to the above computation and subcase3 in case 2, 
$wt_{n-i+1}(\gamma)=q^{|S_{(-a_{l'},-l'+\frac{n}{2}),[\gamma]_{i+1/2}}|}$; $wt_{n-i}=q^{|S_{(\infty,n),[\gamma]_{i+1/2}}|}$.
Moreover, it shines the other half light on the rule why no added box on the first row of $z^{a_l+1}$.
\end{proof}

\begin{proof}

[Theorem ~\ref{mainthm}]

Now we have all components to close the gap!

Let $\tilde{\nu}$ denote  $\{(\nu_{1}+5(p-1))^{a_{1}},...,(\nu_{p-2}+10)^{a_{p-2}},(\nu_{p-1}+5)^{a_{p-1}},(\nu_p)^{a_p}\}$ \text{ where } $\nu=\{\nu_1^{a_1},\nu_2^{a_2},..., \nu_p^{a_p}\}$. Given $ \tau \in T_{\omega}, \omega\in D_{l,\nu}^\lambda$, there exists corresponding $ \tau' \in T_{\omega'}, \omega'\in D_{l,\tilde{\nu}}^{\gamma}$  such that $\omega,\omega'$ have the same added box position and $\tau,\tau'$ have the same numbers filled in. A key observation is that $\gamma$ is not necessarily $\tilde{\lambda}$. Let $\Gamma$ denote the set of all possible $\gamma$'s.
From the definition of weight, $wt(\tau)=wt({\tau'})$. 
We call the number of pair $(i,j)$ where both rows have added boxes to be $re(\tau)$, the redundancy of $\tau$.  $h_{l,\nu}^{\lambda}=\sum_{i=0}^{l/2} \sum_{\omega \in D_{l,\nu}^\lambda} \sum_{ \tau\in T_\omega, re(\tau)=i} wt(\tau).$
For any $\gamma$ with $h_{l,\tilde{\nu}}^{\gamma}\neq 0$, all $\tau$ have the same redundancy, denoted by $i$. Thus,
$h_{-1^\ell,\tilde{\nu}}^{{\gamma}}=\sum_{\omega \in D_{l,\tilde{\nu}}^{{\gamma}}} \sum_{ \tau\in T_\omega, re(\tau)=i} wt(\tau)$.
Therefore, $h_{-1^\ell,\nu}^{\lambda}=\sum_{\gamma\in \Gamma} \sum_{\omega \in D_{l,\nu}^\lambda\text{ and } \omega' \in D_{l,\tilde{\nu}}^\gamma } \sum_{ \tau\in T_\omega} wt(\tau)=\sum_{\gamma\in \Gamma} h_{-1^\ell,\tilde{\nu}}^{\gamma}$.  
Recall the definition of $[w]$ and $W_{l,\nu}^{\lambda},$ each $w$ corresponds to one $\omega\in D_{l,\nu}^{\lambda}$, while the equivalent relation is the rearranging order of $\omega$.
\end{proof}

\begin{corollary}
For any partition $\nu=\{{a_1}^{n_1},{a_2}^{n_2},...{a_j}^{n_j}\}$ with distinct parts differing by at least 2 and $t_1=0$, for any $\ell,\lambda$, each equivalence relation in the set $W_{\ell,\nu}^{\lambda}$ has $\omega_{[i]}^{1,0},\omega_{[i]}^{0,1}(1\leq i\leq \ell)$ fixed.
We denote it by $\Omega_{[i]}^{1,0},\Omega_{[i]}^{0,1}$.
For the simplicity of $h$, we further write \[r=\sum_{i<j} n_i n_j+ \sum_{i<j} n_i(\Omega_{[j]}^{0,1}-\Omega_{[j]}^{1,0})+\sum_i n_i \Omega_{[i]}^{0,1};\]
\[M=\frac{1}{2} (\ell-\sum_i \Omega_{[i]}^{1,0}+\Omega_{[i]}^{0,1});\]
\[m_i= n_i-\Omega_{[a_i]}^{1,0}-\Omega_{[a_i]}^{0,1} \text{ for all i. }\]
Then,
\[h_{-1^\ell,\nu}^{\lambda}=q^r\prod_{i=1}^j\binom{n_i}{ \Omega_{[a_i]}^{1,0},\Omega_{[a_i]}^{0,1}} \sum_{e_1+e_2+...+e_j=M}  q^{-1/2 ((\sum_i m_i)-2M )^2} \prod_{i=1}^j \binom{m_i}{ e_i} q^{1/2 ( m_i^2+2e_i(e_i-m_i) )}.\]
\end{corollary}

\begin{proof}
We have $t_1=0$ if and only if there is no 1-dimensional block. Combining with the condition on gap of size 2, we know that all $\omega_{[i]}^{1,0},\omega_{[i]}^{0,1}$ are the same for all equivalent classes. Thus the first five summands in $r([\omega])$ are the same among all $\omega$. Recall \[\bar{\omega_{[a_i]}}=\binom{n_i}{ \omega_{[a_i]}^{0,1},\omega_{[a_i]}^{1,1},\omega_{[a_i]}^{1,0}}=\binom{n_i } {\omega_{[a_i]}^{0,1},\omega_{[a_i]}^{1,0}}\binom {n_1-\omega_{[a_i]}^{0,1}-\omega_{[a_i]}^{1,0}}{ \omega_{[a_i]}^{1,1}}.\] Thus, summing over all $\omega$ is equivalent to summing over all $\omega_{[a_i]}^{1,1}$. Notice we cannot further simplify the summand over $e_1+e_2+...+e_j=M$ because this is not the $q$-Vandermonde formula $\sum_{t} \binom{m} {k-t}\binom{n} {t}q^{(m-k+t)t}=\binom{m+n}{ k}$.
\end{proof}

Recall in Proposition~\ref{satake} that the ring structure of $\mathcal{H}(G,K)$ is a polynomial ring over $\theta_1,\theta_2,...,\theta_n,\theta_{n}^{-1}.$ Combining with the closed formulas of $h_{-1^\ell,\nu}^\lambda$, any $h_{\mu,\nu}^\lambda$ is theoretically computable. However, due to the complexity nature of $h_{-1^\ell,\nu}^\lambda$ (paired tuple structure), we only give another special case for $h_{\mu,\nu}^\lambda$ in Chapter 4.

\section
{Dual Pieri}
A natural question to ask is what $h_{\mu,\nu}^\lambda$ looks like under the restriction of Dual Pieri rule, i.e., the left multiplying matrix has dominant coweight $\{\ell\}$.
In this chapter, we will study $h_{\{-\ell\},\{a^{n}\}}^{\lambda}$ under the restriction $\ell<a-1$. Theorem~\ref{dualpieri} states: For any $A\in \text{GL}_n ({k}((z)))$ with  $\rho(A)=\{l\}$, any B $\in  \text{GL}_n ({k}((z)))$ with $\sigma(B)=\{a_1,...,a_n\}$, $\sigma(AB)=\{b_1,...,b_n\}$, we have inequalities on $a$'s and $b$'s (Dual Pieri rule):  $a_{i+1}\leq b_i\leq a_{i-1}$. In Theorem~\ref{dual_n}, we make one simple assumption: $n\geq 2 $ over $\nu=\{a^n\}$. Thus $\lambda=\{a+i,a^{n-2},a-j\}$ for $0\leq i,j\leq \ell$ and $i+j\leq \ell$ and $2|\ell-(i+j)$. For other $\lambda$, $h_{\mu,\nu}^\lambda=0.$ 
Theorem~\ref{dual_1} computes $h_{\{-\ell\},\{a^{n}\}}^{\lambda}$ for $n=1.$ In this chapter, we rewrite $p_{\ell}^{i,j}=h_{\{-\ell\},\{a^n\}}^{\{a+i,a^{n-2},a-j\}}$.

\begin{theorem}
\label{dual_n}

    $p_{\ell}^{i,j}=\begin{cases}
    q^{(\ell+i-2)n-\frac{\ell+i+j}{2}+1}[n][n-1](q-1)^2  &i,j\geq 1\text{ and }i+j<\ell\\
    q^{(\ell+i-2)n-\ell+2}[n][n-1](q-1) &i,j\geq 1 \text{ and }i+j=\ell\\
    q^{(2\ell+1)n-(\ell-1)}[n] &j=0 \text{ and }i=\ell\\
    q^{(\ell+i-1)n-\frac{\ell+i}{2}}[n](q-1) &j=0 \text{ and }i<\ell\\
    q^{(\ell-1)n-(\ell-1)}[n] &i=0 \text{ and }j=\ell\\
    q^{(\ell-1)n-\frac{\ell+j}{2}}[n](q-1)&i=0 \text{ and }j<\ell\\
    q^{(\ell-1)n-\frac{\ell}{2}}[n](q-1)&i=j=0\\
    \end{cases}$
There is a symmetry on $i,j$: $p_{\ell}^{i,j}=q^{(i-j)n}p_{\ell}^{j,i}.$
\end{theorem}
\begin{proof}
From  \cite[Chapter~5]{macdonald1998symmetric}, we recall that $\{c_\lambda\}$ forms a $\mathbbm{C}$-basis of $\mathcal{H}(G,K)$. Here is some equalities in $c_\lambda$.
\[c_{\{k-\ell\}}c_{\{-1^k\}}=c_{\{-1^{k-1},k-1-\ell\}}+q^kc_{\{-1^k,k-\ell\}} \text{ for } k<\text{min}(\ell,2n)\]
\[c_{\{-1^{\ell-1}\}}c_{\{-1\}}=[\ell]c_{\{-1^{\ell}\}}+c_{\{-1^{\ell-2},-2\}}.\]
If $\ell>2n$, we notice that $c_{-1^{2n-1},2n-1-\ell}=c_{2n-\ell}$.
For $\ell>2n$, we have \[c_{\{-\ell\}}=c_{\{1-\ell\}}c_{\{-1\}}-q(c_{\{2-\ell\}}c_{\{-1^2\}}-q^2(c_{\{3-\ell\}}c_{\{-1^3\}}-....-q^{2n-1}c_{2n-\ell})).\]
If $\ell\leq 2n$, we have
\[ c_{\{-\ell\}}=c_{\{1-\ell\}}c_{\{-1\}}-q(c_{\{2-\ell\}}c_{\{-1^2\}}-q^2(c_{\{3-\ell\}}c_{\{-1^3\}}-....-q^{\ell-2}(c_{\{-1\}}c_{\{-1^{\ell-1}\}}-[\ell]c_{\{-1^\ell\}}))).\]

From the definition of Hecke module, $c_{\mu}\times d_{\lambda}=\sum_{\nu} h_{\mu\nu}^{\lambda} d_{\nu},$ and the commutativity of $\mathcal{H}(G,K)$, we will prove the equality by induction on $\ell$. The general strategy is: since each $c_{-k}$ acts on $d_{\{a+i,a^{n-2},a-j\}}$,  we write explicitly the summation for $p_\ell^{i,j}$ in terms of $h$ and $p$.

If $\ell\leq 2n$,
\[c_{-\ell}\times d_{\{a+i,a^{n-2},a-j\}}=-[\ell](\prod_{s=1}^{\ell-2}-q^s)c_{\{-1^\ell\}}\times d_{\{a+i,a^{n-2},a-j\}}+\sum_{r=1}^{\ell-1} (\prod_{s=1}^{r-1}-q^s) c_{\{-1^r\}}c_{\{r-\ell\}}\times d_{\{a+i,a^{n-2},a-j\}}\]

\[p_{\ell}^{i,j}=-[\ell](\prod_{s=1}^{\ell-2}-q^s)h_{\{-1^\ell\},\{a^n\}}^{\{a+i,a^{n-2},a-j\}}+\sum_{r=1}^{\ell-1} (\prod_{s=1}^{r-1}-q^s) \sum_{|e-i|,|j-f|\leq 1;|e-i|+|f-j|\leq r;e,f\geq 0}p_{\ell-r}^{e,f} h_{\{-1^r\},\{a+e,a^{n-2},a-f\}}^{\{a+i,a^{n-2},a-j\}}.\]
If $\ell>2n$, \[p_{\ell}^{i,j}=\sum_{r=1}^{2n-1} (\prod_{s=1}^{r-1}-q^s) \sum_{|e-i|,|j-f|\leq 1;|e-i|+|f-j|\leq r;e,f\geq 0}p_{\ell-r}^{e,f} h_{\{-1^r\},\{a+e,a^{n-2},a-f\}}^{\{a+i,a^{n-2},a-j\}}+(\prod_{s=1}^{2n-1}-q^s)  p_{\ell-2n}^{i,j}.\]

For $ i,j\geq 2$, \[p_{\ell}^{i,j}=\sum_{v\in \{-1,0,1\}}\sum_{r=1}^{\text{ min }(2n,\ell)-1}\sum_{u\in \{-1,0,1\}} h_{\{-1^r\},\{a+u+i,a^{n-2},a-v-j\}}^{\{a+i,a^{n-2},a-j\}}  p_{\ell-r}^{u+i,v+j}(\prod_{s=1}^{r-1}-q^s)+ (\prod_{s=1}^{2n-1}-q^s)  p_{\ell-2n}^{i,j}\] (the last term is only nonzero when $\ell\geq 2n+i+j$). To avoid extensive heavy notation, we denote $A_{1}^r,A_{2}^r,...A_{9}^r$ the summand, 
in the order of the following $h$'s. Explicitly, $A_{1}^r=h_{\{-1^{2r}\},\{a+i,a^{n-2},a-j\}}^{\{a+i,a^{n-2},a-j\}}  p_{\ell-2r}^{i,j}(\prod_{s=1}^{2r-1}-q^s)$, $A_{2}^r=h_{\{-1^{2r+2}\},\{a+i-1,a^{n-2},a-j+1\}}^{\{a+i,a^{n-2},a-j\}}  p_{\ell-2r-2}^{i,j}(\prod_{s=1}^{2r+1}-q^s).$
Therefore, we have 
\begin{enumerate}
    \item \[h_{\{-1^{2r}\},\{a+i,a^{n-2},a-j\}}^{\{a+i,a^{n-2},a-j\}}= {n-2\choose r} q^{(n+2-r)r}+ 2{n-2\choose r-1} q^{-1-r^2+n+rn}+{n-2\choose r-2} q^{(n-r)(r+2)}\] ($A_1^r$ has $r\in [1,\frac{\ell-i-j}{2}]$ if $\ell<2n \text{ or } n> \frac{\ell-i-j}{2}$. $A_1^r$ has $r\in [1,n-1]$ otherwise.);

\item\[h_{\{-1^{2r+2}\},\{a+i-1,a^{n-2},a-j+1\}}^{\{a+i,a^{n-2},a-j\}}= {n-2\choose r} q^{n(3+r)-2r-r^2-3}\]($A_2^r$ has $r\in [0,\frac{\ell-i-j}{2}]$ if $\ell<2n \text{ or } n> \frac{\ell-i-j}{2}$. $A_2^r$ has $r\in [0,n-2]$ otherwise.);

\item\[h_{\{-1^{2r+2}\},\{a+i+1,a^{n-2},a-j-1\}}^{\{a+i,a^{n-2},a-j\}}= {n-2\choose r} q^{n(1+r)-(r+1)^2)}\]($A_3^r$ has $r\in [0,\frac{\ell-i-j}{2}-2]$ if $\ell<2n \text{ or } n> \frac{\ell-i-j}{2}$. $A_3^r$ has $r\in [0,n-2]$ otherwise.); 

\item\[h_{\{-1^{2r+2}\},\{a+i+1,a^{n-2},a-j+1\}}^{\{a+i,a^{n-2},a-j\}}= {n-2\choose r} q^{n(1+r)-(r+1)^2)-1}\]($A_4^r$ has $r\in [0,\frac{\ell-i-j}{2}-1]$ if $\ell<2n \text{ or } n> \frac{\ell-i-j}{2}$. $A_4^r$ has $r\in [0,n-2]$ otherwise.);

\item\[h_{\{-1^{2r+2}\},\{a+i-1,a^{n-2},a-j-1\}}^{\{a+i,a^{n-2},a-j\}}= {n-2\choose r} q^{n(3+r)-(r+1)^2)-1}\]($A_5^r$ has $r\in [0,\frac{\ell-i-j}{2}-1]$ if $\ell<2n \text{ or } n> \frac{\ell-i-j}{2}$. $A_5^r$ has $r\in [0,n-2]$ otherwise.); 

\item\[h_{\{-1^{2r+1}\},\{a+i,a^{n-2},a-j+1\}}^{\{a+i,a^{n-2},a-j\}}= {n-2\choose r} q^{(r+1)n-r^2-1}+{n-2\choose r-1} q^{n(2+r)-(r+1)^2-1}\]($A_6^r$ has $r\in [0,\frac{\ell-i-j}{2}]$ if $\ell<2n \text{ or } n> \frac{\ell-i-j}{2}$. $A_6^r$ has $r\in [0,n-1]$ otherwise.);

\item\[h_{\{-1^{2r+1}\},\{a+i,a^{n-2},a-j-1\}}^{\{a+i,a^{n-2},a-j\}}= {n-2\choose r} q^{(r+1)n-r^2}+{n-2\choose r-1} q^{n(2+r)-(r+1)^2}\]($A_7^r$ has $r\in [0,\frac{\ell-i-j}{2}-1]$ if $\ell<2n \text{ or } n> \frac{\ell-i-j}{2}$. $A_7^r$ has $r\in [0,n-1]$ otherwise.);

\item\[h_{\{-1^{2r+1}\},\{a+i+1,a^{n-2},a-j\}}^{\{a+i,a^{n-2},a-j\}}= {n-2\choose r} q^{rn-r^2}+{n-2\choose r-1} q^{n(1+r)-(r+1)^2}\]($A_8^r$ has $r\in [0,\frac{\ell-i-j}{2}-1]$ if $\ell<2n \text{ or } n> \frac{\ell-i-j}{2}$. $A_8^r$ has $r\in [0,n-1]$ otherwise.);

\item\[h_{\{-1^{2r+1}\},\{a+i-1,a^{n-2},a-j\}}^{\{a+i,a^{n-2},a-j\}}= {n-2\choose r} q^{(r+2)n-r^2-1}+{n-2\choose r-1} q^{n(3+r)-(r+1)^2-1}\]($A_9^r$ has $r\in [0,\frac{\ell-i-j}{2}]$ if $\ell<2n \text{ or } n> \frac{\ell-i-j}{2}$. $A_9^r$ has $r\in [0,n-1]$ otherwise.);
\end{enumerate}

If $ 0\leq r\leq \text{ min }(\frac{\ell-i-j}{2},n-1)$,
\[A_{2}^r+A_{4}^{r-1}=-[n][n-1](q-1)^2q^{n(\ell+i-r-2)-r^2-\frac{\ell+i+j}{2}+1}{n\choose r}-[n][n-1](q-1)^2q^{n(\ell+i-r-1)-(r+1)^2-\frac{\ell+i+j}{2}+1}{n\choose r-1}\]
\[=-A_{6}^{r}.\]
(If $r=n-1$, it follows that $A_{2}^r=0.$) 
For $0\leq r\leq  \text{ min }(\frac{\ell-i-j}{2}-1,n-2)$, $A_{3}^{r-1}+A_{5}^{r}+A_{7}^r=0$. For $1\leq r\leq\text{ min }(\frac{\ell-i-j}{2},n-1)$, $A_{1}^{r}+A_{8}^{r-1}+A_{9}^r=0$. We combine all summands under the assumption $\ell\leq 2n$ or $n> \frac{\ell-i-j}{2}$, the only left over term is $A_{9}^0$. 
We combined all the summands under the assumption $\ell> 2n$ and $n\leq \frac{\ell-i-j}{2}$, thus left over terms are $A_8^{n-1},A_9^0, (\prod_{s=1}^{2n-1}-q^s)  p_{\ell-2n}^{i,j}$. From the induction, $(\prod_{s=1}^{2n-1}-q^s)  p_{\ell-2n}^{i,j}=A_8^{n-1}$. Thus, the left over term is $A_{9}^0$.
$p_{\ell}^{i,j}=p_{\ell-1}^{i-1,j}q^{2n-1}$. 

For $i\geq 2,j=1$, \[p_{\ell}^{i,j}=\sum_{v\in \{-1,0,1\}}\sum_{r=1}^{\text{ min }(2n,\ell)-1}\sum_{u\in \{-1,0,1\}} h_{\{-1^r\},\{a+u+i,a^{n-2},a-v-j\}}^{\{a+i,a^{n-2},a-j\}}  p_{\ell-r}^{u+i,v+j}(\prod_{s=1}^{r-1}-q^s)+ (\prod_{s=1}^{2n-1}-q^s)  p_{\ell-2n}^{i,j}.\] We denote $B_{1}^r,B_{2}^r,...B_{9}^r$ the summands, the same as that of $i,j\geq 2$.
Therefore, we have,
\begin{enumerate}
    \item \begin{multline*}
        h_{\{-1^{2r}\},\{a+i,a^{n-2},a-1\}}^{\{a+i,a^{n-2},a-1\}}= {n-2\choose r} q^{(n+2-r)r}+ 2{n-2\choose r-1} q^{-1-r^2+n+rn}\\+{n-2\choose r-2} q^{(n-r)(r+2)}+{n-2\choose r-1,1} q^{nr-r^2-r}+{n-2\choose r-2,1} q^{(n-r)(r+1)}\end{multline*}

\item \[h_{\{-1^{2r+2}\},\{a+i-1,a^{n-1}\}}^{\{a+i,a^{n-2},a-1\}}= {n-1\choose 1,r} q^{n(2+r)-(r+1)^2}\]

\item \[h_{\{-1^{2r+2}\},\{a+i+1,a^{n-1}\}}^{\{a+i,a^{n-2},a-j\}}= {n-1\choose r,1} q^{nr-r^2-2r}\]

\item \[h_{\{-1^{2r+2}\},\{a+i+1,a^{n-2},a-j-1\}}^{\{a+i,a^{n-2},a-j\}}= {n-2\choose r} q^{n(r+1)-(r+1)^2)}\]

\item \[h_{\{-1^{2r+2}\},\{a+i-1,a^{n-2},a-j-1\}}^{\{a+i,a^{n-2},a-j\}}= {n-2\choose r} q^{n(3+r)-(r+1)^2)-1}\]

\item \[h_{\{-1^{2r+1}\},\{a+i,a^{n-1}\}}^{\{a+i,a^{n-2},a-j\}}= {n-1\choose r,1} q^{rn-r^2}+{n-1\choose 1, r-1} q^{n(1+r)-(r+1)^2+1}\]

\item \[h_{\{-1^{2r+1}\},\{a+i,a^{n-2},a-j-1\}}^{\{a+i,a^{n-2},a-j\}}= {n-2\choose r} q^{(r+1)n-r^2}+{n-2\choose r-1} q^{n(2+r)-(r+1)^2}\]

\item \[h_{\{-1^{2r+1}\},\{a+i+1,a^{n-2},a-j\}}^{\{a+i,a^{n-2},a-j\}}= {n-2\choose r} q^{rn-r^2}+{n-2\choose r-1} q^{n(1+r)-(r+1)^2}+{n-2\choose r-1,1} q^{nr-(r+1)r}\]

\item \[h_{\{-1^{2r+1}\},\{a+i-1,a^{n-2},a-j\}}^{\{a+i,a^{n-2},a-j\}}= {n-2\choose r} q^{(r+2)n-r^2-1}+{n-2\choose r-1} q^{n(3+r)-(r+1)^2-1}+{n-2\choose r-1,1} q^{n(r+2)-(r+1)r-1}.\]
\end{enumerate}
Before computing $p_{\ell}^{i,1}$, one key observation on $h$'s is the symmetry between $(i,1)$ and $(1,i);$ $h_{\{-1^{2r}\},\{a+i,a^{n-2},a-1\}}^{\{a+i,a^{n-2},a-1\}}=h_{\{-1^{2r}\},\{a+1,a^{n-2},a-i\}}^{\{a+1,a^{n-2},a-i\}}$ (all five terms are the same.) After direct computation for all nine $B$s, we observe that $p_{\ell}^{i,1}=q^{(i-1)n}p_{\ell}^{1,i}$.

If $ 0\leq r\leq \text{ min }(\frac{\ell-i-1}{2},n-2)$, $B_{2}^r+B_{3}^{r-1}+B_{6}^{r}=0$. If $ 0\leq r\leq \text{ min }(\frac{\ell-i-1}{2}-1,n-2)$, $B_{4}^{r-1}+B_{5}^{r}+B_{7}^r=0$. If $ 1\leq r\leq \text{ min }(\frac{\ell-i-1}{2},n-1)$, $B_{1}^{r}+B_{8}^{r-1}+B_{9}^r=0$. 
We combine all summands under the assumption $\ell\leq 2n$ or $n> \frac{\ell-i-j}{2}$, and the only left over term is $B_{9}^0$.  We combined all the summands under the assumption $\ell> 2n$ and $n\leq \frac{\ell-i-j}{2}$, and left over terms are $B_8^{n-1},B_9^0, (\prod_{s=1}^{2n-1}-q^s)  p_{\ell-2n}^{i,j}$. From the induction, $(\prod_{s=1}^{2n-1}-q^s)  p_{\ell-2n}^{i,j}=B_8^{n-1}$. Thus, the left over term is $B_{9}^0$.

For $i\geq 2,j=0$, \[p_{\ell}^{i,j}=\sum_{v\in \{-1,0,1\}}\sum_{r=1}^{\text{ min }(2n,\ell)-1}\sum_{u\in \{-1,0,1\}} h_{\{-1^r\},\{a+u+i,a^{n-2},a-v-j\}}^{\{a+i,a^{n-2},a-j\}}  p_{\ell-r}^{u+i,v+j}(\prod_{s=1}^{r-1}-q^s)+ (\prod_{s=1}^{2n-1}-q^s)  p_{\ell-2n}^{i,j}.\]
We denote $C_{1}^r,C_{2}^r,...C_{6}^r$ the summands, the same as before. Therefore, we have
\begin{enumerate}
    \item \[h_{\{-1^{2r}\},\{a+i,a^{n-1}\}}^{\{a+i,a^{n-1}\}}= {n-1\choose r} q^{(n+1-r)r}+ {n-1\choose r-1} q^{(n-r)(1+r)}\]
\item \[h_{\{-1^{2r+2}\},\{a+i-1,a^{n-2},a-1\}}^{\{a+i,a^{n-1}\}}= {n-2\choose r} q^{n(3+r)-(r+1)^2-1}\]

\item \[h_{\{-1^{2r+2}\},\{a+i+1,a^{n-2},a-1\}}^{\{a+i,a^{n-1}\}}= {n-2\choose r} q^{n(r+1)-(r+1)^2}\]

\item \[h_{\{-1^{2r+1}\},\{a+i,a^{n-2},a-1\}}^{\{a+i,a^{n-1}\}}= {n-2\choose r} q^{(r+1)n-r^2}+{n-2\choose r-1} q^{n(2+r)-(r+1)^2}\]
\item \[h_{\{-1^{2r+1}\},\{a+i-1,a^{n-1}\}}^{\{a+i,a^{n-1}\}}= {n-1\choose r} q^{(r+2)n-r^2-r-1}\]

\item \[h_{\{-1^{2r+1}\},\{a+i+1,a^{n-1}\}}^{\{a+i,a^{n-1}\}}= {n-1\choose r} q^{rn-r^2-r}.\]
\end{enumerate}
Before computing $p_{\ell}^{i,0}$, one key observation on $h$'s is the symmetry between $(i,0)$ and $(0,i).$ After direct computation for all six $C$s, we observe that $p_{\ell}^{i,0}=q^{in}p_{\ell}^{0,i}$.

If $ 0\leq r\leq \text{ min }(\frac{\ell-i}{2}-1,n-2)$, $C_{2}^r+C_{3}^{r-1}+C_{4}^{r}=0$. If $ 1\leq r\leq \text{ min }(\frac{\ell-i}{2},n-1)$, $C_{1}^{r}+C_{5}^{r}+C_{6}^{r-1}=0$.
We combine all summands under the assumption $\ell\leq 2n$ or $n> \frac{\ell-i-j}{2}$, and the only left over term is $C_{5}^0$. We combined all the summands under the assumption $\ell> 2n$ and $n\leq \frac{\ell-i-j}{2}$, and left over terms are $C_6^{n-1},C_5^0, (\prod_{s=1}^{2n-1}-q^s)  p_{\ell-2n}^{i,j}$. From the induction, $(\prod_{s=1}^{2n-1}-q^s)  p_{\ell-2n}^{i,j}=C_6^{n-1}$. Thus, the left over term is $C_{5}^0$. $p_{\ell}^{i,0}=p_{\ell-1}^{i-1,0}q^{2n-1}$.

For $i=1,j=0$,
\begin{multline*}
    p_{\ell}^{i,j}=\sum_{v\in \{-1,0,1\}}\sum_{r=1}^{\text{ min }(2n,\ell)-1}\sum_{u\in \{-1,0,1\}} h_{\{-1^r\},\{a+u+i,a^{n-2},a-v-j\}}^{\{a+i,a^{n-2},a-j\}}  p_{\ell-r}^{u+i,v+j}(\prod_{s=1}^{r-1}-q^s) \\-[\ell](\prod_{s=1}^{\ell-2}-q^s)h_{\{-1^\ell\},\{a^n\}}^{\{a+i,a^{n-2},a-j\}}+(\prod_{s=1}^{2n-1}-q^s)  p_{\ell-2n}^{i,j}.
\end{multline*}
Therefore, it follows that, 
\begin{enumerate}
    \item \[h_{\{-1^{2r}\},\{a+1,a^{n-1}\}}^{\{a+1,a^{n-1}\}}= {n-1\choose r} q^{(n+1-r)r}+ {n-1\choose r-1} q^{(n-r)(1+r)}+{n-1\choose r-1,1} q^{(n-r)r}\]

\item \[h_{\{-1^{2r+2}\},\{a^{n-1},a-1\}}^{\{a+1,a^{n-1}\}}= {n-1\choose r,1} q^{n(2+r)-(r+1)^2+1}\]

\item \[h_{\{-1^{2r+2}\},\{a+2,a^{n-2},a-1\}}^{\{a+1,a^{n-1}\}}= {n-2\choose r} q^{n(r+1)-(r+1)^2}\]

\item \[h_{\{-1^{2r+1}\},\{a+1,a^{n-2},a-1\}}^{\{a+1,a^{n-1}\}}= {n-2\choose r} q^{(r+1)n-r^2}+{n-2\choose r-1} q^{n(2+r)-(r+1)^2}+{n-2\choose r-1,1} q^{(1+r)n-r-r^2}\]

\item \[h_{\{-1^{2r+1}\},\{a^{n}\}}^{\{a+1,a^{n-1}\}}= {n\choose r,1} q^{(r+1)n-(r+1)r}\]

\item \[h_{\{-1^{2r+1}\},\{a+2,a^{n-1}\}}^{\{a+1,a^{n-1}\}}= {n-1\choose r} q^{rn-r^2-r}.\]
\end{enumerate}
Instead of computing $p_{\ell}^{i,1}$, one key observation on $h$'s is the symmetry between $(1,0)$ and $(0,1).$ After direct computation for all six terms, we observe that $p_{\ell}^{1,0}=q^np_{\ell}^{0,1}$. Recall the measure of \[K\pi^{\ell}K(=\mu(c_{\ell})=\mu(c_{-\ell}))=q^{(2n-1)(\ell-1)}[2n],\] we can write $\sum_{i,j} p_{\ell}^{i,j}=q^{(2n-1)(\ell-1)}[2n].$ 

We notice that $\sum_{i}p_{\ell}^{i+2m-1,i}=[n]^2q^{(n-1)(\ell-2)+(\frac{\ell-1}{2}+m)n-1}(q-1)$ for $2m-1\neq \ell,-\ell,1,-1$.
$\sum_{i}p_{\ell}^{i+\ell,i}=p_{\ell}^{\ell,0}=[n]q^{(n-1)(\ell-1)+\ell n}$.
$\sum_{i}p_{\ell}^{i-\ell,i}=p_{\ell}^{0,\ell}=[n]q^{(n-1)(\ell-1)}$. 
Combining with the measure of $K\pi^\ell K$, 
we have $p_{\ell}^{1,0}=q^{\ell n-\frac{\ell +1}{2}}[q](q-1)$ and $ p_{\ell}^{0,1}=q^{\ell n-\frac{\ell -1}{2}}[q](q-1)$.

For $i=j=0$, 
\begin{multline*}
    p_{\ell}^{i,j}=\sum_{v\in \{-1,0,1\}}\sum_{r=1}^{\text{ min }(2n,\ell)-1}\sum_{u\in \{-1,0,1\}} h_{\{-1^r\},\{a+u+i,a^{n-2},a-v-j\}}^{\{a+i,a^{n-2},a-j\}}  p_{\ell-r}^{u+i,v+j}(\prod_{s=1}^{r-1}-q^s)\\ -[\ell](\prod_{s=1}^{\ell-2}-q^s)h_{\{-1^\ell\},\{a^n\}}^{\{a+i,a^{n-2},a-j\}}+(\prod_{s=1}^{2n-1}-q^s)  p_{\ell-2n}^{i,j}. 
\end{multline*}
We denote $D_{1}^r,D_{2}^r,D_{3}^r,D_{4}^r$ those summands, the same as before.
It follows that 
\begin{enumerate}
    \item \[h_{\{-1^{2r}\},\{a^{n}\}}^{\{a^{n}\}}= {n\choose r} q^{nr-r^2}\]

\item \[h_{\{-1^{2r+2}\},\{a+1,a^{n-1},a-1\}}^{\{a^{n}\}}= {n-2\choose r} q^{n(1+r)-(r+1)^2}\]

\item \[h_{\{-1^{2r+1}\},\{a+1,a^{n-1}\}}^{\{a^{n}\}}= {n-1\choose r} q^{rn-r^2-r}\]

\item \[h_{\{-1^{2r+1}\},\{a^{n-1},a-1\}}^{\{a^{n}\}}= {n-1\choose r} q^{n(r+1)-r^2-r}.\]
\end{enumerate}
If $ 0\leq r\leq \text{ min }(\frac{\ell}{2}-1,n-1)$, $D_{3}^r=D_4^r$. If $ 1\leq r\leq \text{ min }(\frac{\ell}{2}-2,n-2)$, $D_{1}^{r}+D_{2}^{r}=D_{3}^{r-1}+D_{4}^{r}$ (from the equality $[n-r][r](q-1)+[n]=q^r[n-r]+q^{n-r}[r]$). Thus, we have $D_{4}^{\frac{\ell}{2}-2}+2D_{3}^{\frac{\ell}{2}-1}+D_{1}^{\frac{\ell}{2}}+D_{1}^{\frac{\ell}{2}-1}+D_{2}^{0}=0$.
We combine all summands under the assumption $\ell\leq 2n$ or $n> \frac{\ell-i-j}{2}$, and the only left over term is $D_{4}^0$. $p_{\ell}^{0,0}=p_{\ell-1}^{0,1}q^{n}$.
We combined all the summands under the assumption $\ell> 2n$ and $n\leq \frac{\ell-i-j}{2}$: with $D_{4}^{n-2}+2D_{3}^{n-1}+D_{1}^{n-1}+D_{2}^{0}+(\prod_{s=1}^{2n-1}-q^s)  p_{\ell-2n}^{0,0}=0$, and the left over term is $D_{4}^0$.

For $i=j=1$, $p_{\ell}^{1,1}$ is obtained from: $\sum_{i,j} p_{\ell}^{i,j}=q^{(2n-1)(\ell-1)}[2n].$ The proof is the same as the one for $i=1,j=0$, by evaluating $\sum_{i,j} p_{\ell}^{i,j}$. 
\end{proof}
We now analyze $h_{\{-\ell\},\{a\}}^{a+j}$ , i.e., $n=1$.
\begin{theorem}
\label{dual_1}
$h_{\{-\ell\},\{a\}}^{a+j}=\begin{cases}
   q^{\frac{j+\ell}{2}-1}(q-1) &j=\ell-2,\ell-4,...,2-\ell\\ 
   1 &j=-\ell \\
   q^{\ell} &j=\ell\\
\end{cases}$
\end{theorem}
\begin{proof}
Claim:
For all $ A \in GL_2(k[[z]])$ with $\rho(A)=a$, we have $KAK=\bigsqcup_{X\in\mathcal{A}} KX$ where $\mathcal{A}$ consists of upper triangular matrices $X$ with following properties: 
\begin{enumerate}
    \item $X_{1,1}=z^t,X_{2,2}=z^r$
    \item $X_{1,2}=b_0+b_1 z+b_2z^2+...+b_{r-1}z^{r-1} $  where $b_i\in k$
    \item if $t\neq 0,\text{ then }  b_0\neq 0$
    \item $t+r=\ell$
\end{enumerate}  

The proof of the claim follows the same strategy of Lemma ~\ref{decomp_A}. 

From trick 1 of Proposition~\ref{young} in Chapter 3, we have the following: each $X\in \mathcal{A} \text{ with } X_{2,2}=z^{\frac{\ell+j}{2}},X_{1,1}=z^{\frac{\ell-j}{2}}$, $X\begin{pmatrix}
    1&z\\
    0&z^a\\
\end{pmatrix}$ is in the form of $z^{a+j}$. The number of X in $\mathcal{A}$ with the above condition is $q^{\frac{j+\ell}{2}-1}(q-1)$ if $j\neq -\ell,\ell$, and $h_{\{-\ell\},\{a\}}^{a+\ell}=q^{\ell},h_{\{-\ell\},\{a\}}^{a-\ell}=1$.
\end{proof}

\printbibliography[heading=bibintoc]

@article{rains2019birational,
  title={The birational geometry of noncommutative surfaces},
  author={Rains, Eric. M},
  journal={arXiv preprint arXiv:1907.11301},
  year={2019}
}

@article{rains2016noncommutative,
  title={The noncommutative geometry of elliptic difference equations},
  author={Rains, Eric. M},
  journal={arXiv preprint arXiv:1607.08876},
  year={2016}
}

@article{rains2013generalized,
  title={Generalized Hitchin systems on rational surfaces},
  author={Rains, Eric. M},
  journal={arXiv preprint arXiv:1307.4033},
  year={2013}
}

@article{bump,
  title={Hecke Algebras},
  author={Bump, Daniel},
  journal={http://sporadic.stanford.edu/bump/math263/hecke.pdf},
  year={2010}
}

@article{rains2007vanishing,
  title={Vanishing integrals of Macdonald and Koornwinder polynomials},
  author={Rains, Eric. M and Vazirani, Monica},
  journal={Transformation Groups},
  volume={12},
  pages={725--759},
  year={2007},
  publisher={Springer}
}

@article{venkateswaran2012vanishing,
  title={Vanishing integrals for Hall--Littlewood polynomials},
  author={Venkateswaran, Vidya},
  journal={Transformation Groups},
  volume={17},
  pages={259--302},
  year={2012},
  publisher={Springer}
}

@book{macdonald1998symmetric,
  title={Symmetric functions and Hall polynomials},
  author={Macdonald, Ian Grant},
  year={1998},
  publisher={Oxford University Press}
}

@book{hahn2013classical,
  title={The classical groups and K-theory},
  author={Hahn, Alexander. J and O'Meara, O. Timothy},
  volume={291},
  year={2013},
  publisher={Springer Science \& Business Media}
}

@article{marcus1990determinants,
  title={Determinants of sums},
  author={Marcus, Marvin},
  journal={The College Mathematics Journal},
  volume={21},
  number={2},
  pages={130--135},
  year={1990},
  publisher={Taylor \& Francis}
}
\end{document}